\documentclass[12pt]{amsart}
\usepackage[margin=1.4in]{geometry}
\usepackage[a4paper, left=1.3in, right=.3in]{}
\usepackage{amsmath,amssymb,amsthm,graphicx,amsxtra,setspace}
\usepackage[utf8]{inputenc}
\usepackage{mathrsfs}
\usepackage{hyperref}
\usepackage{xcolor}
\usepackage{upgreek}
\usepackage{mathtools}
\allowdisplaybreaks

\newtheorem{theorem}{Theorem}[section]

\newtheorem{definition}{Definition}[section]

\theoremstyle{definition}%

\usepackage{latexsym}
\usepackage{hyperref}
\usepackage{graphicx}
\usepackage{epstopdf}
\usepackage{bigints}
\usepackage{siunitx}

\usepackage{amssymb,bbm,enumerate,bbm,amsmath}
\usepackage{changepage}

\usepackage{mathrsfs}
\usepackage{tikz}
\usepackage{ dsfont}
\usepackage{hyperref}
\usetikzlibrary{arrows}
\usepackage{amsthm}
\usepackage{array,amsfonts}
\usepackage{amsmath}
\usepackage[utf8]{inputenc}
\usepackage[T1]{fontenc}
\usepackage{mathtools}
\usepackage{enumitem}
\usepackage[thinc]{esdiff}
\usepackage{multirow}
\usepackage{graphics}
\usepackage{amsmath,bm}

\newtheorem{remark}[theorem]{Remark}
\newtheorem{example}[theorem]{Example}
\newtheorem{lemma}[theorem]{Lemma}
\newtheorem{corollary}[theorem]{Corollary}

\allowdisplaybreaks

\let\originalleft\left
\let\originalright\right
\renewcommand{\left}{\mathopen{}\mathclose\bgroup\originalleft}
\renewcommand{\right}{\aftergroup\egroup\originalright}

\newcommand{\Addresses}{{
		\footnote{

			\noindent	 \textsuperscript{1,2} Department of Mathematics, Indian Institute of Technology Roorkee, Roorkee, 247667, India.	
			
			\noindent  \textit{e-mail\textsuperscript{1}:} \texttt{p\_yadav@ma.iitr.ac.in}
			
			\noindent  \textit{e-mail\textsuperscript{2}:} \texttt{tanuja.srivastava@ma.iitr.ac.in}

}}}

\begin{document}
	\title[]{Maximum likelihood estimation in the multivariate and matrix variate symmetric Laplace distributions through group actions \Addresses}
	\author [Pooja Yadav.  Tanuja Srivastava]{Pooja Yadav\textsuperscript{1}.  Tanuja Srivastava\textsuperscript{2}}

	\begin{abstract}
	In this paper, we study the maximum likelihood estimation of the parameters of the multivariate and matrix variate symmetric Laplace distributions through group actions. The multivariate and matrix variate symmetric Laplace distributions are not in the exponential family of distributions. We relate the maximum likelihood estimation problems of these distributions to norm minimization over a group and build a correspondence between stability of data with respect to the group action and the properties of the likelihood function.
	\end{abstract}
	\maketitle
	
	\section{Introduction}\label{sec1}  
		The main challenge in statistics is to find a statistical model that best fit for the observed data. Maximum likelihood estimation is a fundamental approach that gives most likely the value of the parameter in the assumed statistical distribution given the observed data \cite{Aldrich}. The value of the parameter that maximizes the likelihood (log-likelihood) function is called the maximum likelihood estimator (MLE) of the parameter. There are many algorithms for computing the MLEs (see e.g., \cite{GRL}, \cite{myung}, \cite{MK}), but there is a growing interest in understanding the conditions under which an MLE exists and whether it is unique. In this paper, we will discuss such types of questions for the Laplace distribution using the concept of group actions.
	
	The Laplace distribution has sharp peaks at the location parameter and heavy tails. The Laplace distribution has many real-world applications in signal processing, financial data, biological sciences and engineering sciences \cite{FM}, \cite{KKP} and \cite{KP}. In general, the normal distribution is used to fit the empirical data, but in some cases, the Laplace distribution provides better fits for the empirical data than the normal distribution, such as real financial data, where these types of data usually have sharp peaks and heavy tails. The multivariate and matrix variate symmetric Laplace distributions are extensions of the univariate Laplace distribution to the multivariate and matrix variate settings, respectively. In the symmetric Laplace distribution, the location parameter is always assumed to be zero. Estimation of the parameters of the multivariate Laplace distribution is studied by using the other method in \cite{EKL}, \cite{KS} and \cite{Vi}. The maximum likelihood estimation of the parameters of the multivariate and matrix variate symmetric Laplace distributions are studied in \cite{YS}. For more details about these distributions, see \cite{KKP}, \cite{KMP}, \cite{YS} and \cite{Y}. Here, we will talk about the existence and uniqueness of maximum likelihood estimators of the multivariate and matrix variate symmetric Laplace distributions through group actions.
	
	Améndola et al. \cite{amendola2021invariant} discovered the connection between the maximum likelihood estimation and norm minimization over a group orbit for Gaussian models, which are continuous statistical models. Also, they have connected the maximum likelihood estimation and norm minimization for the log-linear model, a discrete statistical model in \cite{AKRS}. Both these statistical models belong to the exponential family of distributions. The study of the group action on a set or space focuses on identifying the properties that remain unchanged. The notions of stability are studied using the minimum norm along its orbits, the orbit of a point consists of all the points obtained by acting the group elements on that point. If the orbit is closed, then the minimum norm is attained, otherwise the minimum norm is attained in the closure of the orbit. We aim to connect the maximum likelihood estimation and stability of data with respect to the group action for the multivariate and matrix variate symmetric Laplace distribution, the non-exponential family of distributions. 
	
	In this paper, we define the symmetric Laplace distributions as the Laplace group model by taking the concentration matrix $\bm{\Psi}= \bm{\Sigma}^{-1}$ of the form $A^\top A$, where $A$ lies in the group $G$. Then, we show that the maximum likelihood estimation problems for the multivariate and matrix variate symmetric Laplace distributions are equivalent to norm minimizations over the group $G$. Finally, we build a connection between the stability of data with respect to group action and the properties of the likelihood function of the Laplace group model.

The paper is organized as follows: in section $2$, the preliminaries results about linear group action and stability with respect to this action are given. In section $3$, the multivariate symmetric Laplace distribution is defined as a group model, and demonstrates that the problem of maximum likelihood estimation for the Laplace group model is equivalent to minimizing the norm over the group. Then, we establish a connection between the properties of the likelihood function and stability of data with respect to the group action. In section $4$, we build the correspondence between the properties of the likelihood function of the matrix variate symmetric Laplace distribution and stability of data with respect to the group action. Finally, section $5$ concludes the paper.

	\section{Preliminaries}
	\noindent\textbf{\emph{Notations}.} The following notations are used throughout the paper: $\mathcal{N}_{p}(0,\bm{\Sigma})$ denotes the $p$-dimensional multivariate normal distribution, where $0$ is a $p$-dimensional vector with zero entries, and $\bm{\Sigma}$ is a $p \times p$ positive definite matrix. $\operatorname{tr}(\bm{A})$ and $\begin{vmatrix} \bm{A}	\end{vmatrix}$ denotes the trace and  the determinant of the matrix $\bm{A}$, respectively. $\bm{A}\otimes \bm{B}$ denotes the Kronecker product of matrices $\bm{A}$ and $\bm{B}$. $\bm{A}^\top$ denotes the transpose of the matrix $\bm{A}$. The notation $\mathcal{MN}_{p,q}(\bm{0}, \bm{\Sigma}_{1}, \bm{\Sigma}_{2})$ is used for matrix variate normal distribution, where $\bm{0}$ is a matrix of order $p\times q$ with all zero entries and $\bm{\Sigma}_{1}, \bm{\Sigma}_{2}$ are positive definite matrices of order $p \times p$ and $q\times q$, respectively. The notation ${Exp}(1)$ is used for the univariate exponential distribution with location parameter $0$ and scale parameter $1$. The notation $\mathcal{SL}_{p}(\bm{\Sigma})$ and  $\mathcal{MSL}_{p,q}(\bm{\Sigma}_{1}, \bm{\Sigma}_{2})$ are used for $p$-dimensional symmetric Laplace distribution and $p\times q$-dimensional matrix variate symmetric Laplace distribution, respectively. $GL(V)$ denotes the general linear group on a vector space $V$. $PD_{p}(\mathbb{K})$ denotes the cone of $p\times p$ positive definite matrices over the field $\mathbb{K}$. $\|.\|$ denotes Euclidean norm for vectors and Frobenius norm for matrices.

	In this section, a linear group action on a vector space and notions of stability with respect to this action have been taken as such, based on the paper of Améndola et al. \cite{amendola2021invariant}.
	\subsection{Stability with respect to the group action} A linear action on a real or complex vector space that corresponds to a representation $\rho: G \mapsto GL_{p}(\mathbb{K})$, that is, each  group element $g\in G$ is assigned to an invertible matrix in $GL_{p}(\mathbb{K})$, where $\mathbb{K}$ is $\mathbb{R}$ or $\mathbb{C}$. The action of $G$ on $\mathbb{K}^p $ defined as:
	\begin{align*}
		\cdot \bm{\colon} G\times \mathbb{K}^p & \to \mathbb{K}^p \\
		g\cdot v &\longmapsto \rho(g)v
	\end{align*}
	The group element  $g\in G$ acts on $\mathbb{K}^p$ by left multiplication with the matrix $\rho(g)$. 
	
	Let $v\in \mathbb{K}^{p}$. The orbit and stabilizer of $v$ are defined as: the orbit of $v$ is $G\cdot v=\{g\cdot v \mid  g\in G\}=\{\rho(g)v \mid g\in G\} \subseteq \mathbb{K}^p$, and the stabilizer of $v$ is $G_{v}= \{g\in G \mid  g\cdot v=v\} \subseteq G$. The orbit of $v$ and its closure depends only on the group $\rho (G)$. Thus $G\subseteq GL_{p}(\mathbb{K})$ can be assumed after restricting to the image of $\rho$. $G\subseteq GL_{p}(\mathbb{K})$  is called Zariski closed, if $G$ is the zero locus of a set of polynomials in the matrix entries, that is, $G\subseteq GL_{p}(\mathbb{K}) \subseteq (\mathbb{K}^p)^2$. $G$ is self-adjoint, if $g\in G$ implies that $g^\top \in G$ ($\mathbb{K}=\mathbb{R}$), and $g^* \in G$ $(\mathbb{K}=\mathbb{C}$).

	\begin{definition}[Stability]\cite{amendola2021invariant}
		Let $v\in \mathbb{K}^p$. Then, $v$ is called
		\begin{enumerate}
			\item unstable, if $ 0 \in$ $\overline{G\cdot v}$, i.e. $\inf\limits_{g\in G} \|g\cdot v\|^2=0$;
			\item semistable, if $0 \notin$ $\overline{G\cdot v}$, i.e. $\inf\limits_{g\in G} \|g\cdot v\|^2 >0$;
			\item polystable, if $v\not= 0$ and $G\cdot v$ is closed;
			\item stable, if $v$ is polystable and $G_v$ is finite.
		\end{enumerate}
	\end{definition} 
	
	Let $G\subseteq GL_{p}(\mathbb{K})$ be a Zariski closed self-adjoint subgroup. For each vector $v\in \mathbb{K}^p$, the map $\gamma_{v}$ be defined as
	\begin{align*}
		\gamma_{v} \bm{\colon} G & \to \mathbb{R}\\
		g & \longmapsto \|g\cdot v\|^2.
	\end{align*}
	Since $G$ is Zariski closed, that is, $G$ is defined by polynomial equations, so we can consider its tangent space $T_{\bm{I}_p}G \subseteq \mathbb{K}^{p \times p}$ at the identity matrix $\bm{I}_{p}$ and we can compute differential of the map $\gamma_{v}$ at identity
	\begin{align*}
		D_{\bm{I}_p}\gamma_{v} \bm{\colon} T_{\bm{I}_p} G & \to \mathbb{R}\\
		\dot{g}& \longmapsto 2 Re[\operatorname{tr}(\dot{g}vv^*)].
	\end{align*}
	The Moment map $\mu $ assigns this differential to each vector $v$, that is, 
	\begin{align*}
		\mu \bm{\colon} \mathbb{K}^p & \to Hom(T_{\bm{I}_p}G, \mathbb{R})\\
		v& \longmapsto D_{\bm{I}_p}\gamma_{v}.
	\end{align*}
	The moment map vanishes at a vector $v$ if and only if the identity matrix $\bm{I}_{p}$ is a critical point of the map $\gamma_{v}$.
	
	\begin{theorem}[Kempf-Ness] \cite{amendola2021invariant}
		\label{thm:1}
		Let $G\subseteq GL_{p}(\mathbb{K})$ be a Zariski closed self-adjoint subgroup with moment map $\mu$, where $\mathbb{K}\in \{\mathbb{R}, \mathbb{C}\}$. If $\mathbb{K}=\mathbb{R}$, let $K$ be the set of orthogonal matrices in $G$. If $\mathbb{K}=\mathbb{C}$, let $K$ be the set of unitary matrices in $G$. For $v\in \mathbb{K}^p$,
		\renewcommand{\labelenumi}{(\alph{enumi})}
		\begin{enumerate}
			\item The vector $v$ is of minimal norm in its orbit if and only if $\mu(v)=0$.
			\item If $\mu(v)=0$ and $w\in G\cdot v$ is such that $\|v\|=\|w\|$, then $w\in K\cdot v$.
			\item If the orbit $G\cdot v$ is closed, then there exists some $w\in G\cdot v$ with $\mu(w)=0$.
			\item If $\mu(v)=0$, then the orbit $G\cdot v$ is closed.
			\item The vector $v$ is polystable if and only if there exists $0\ne w \in G\cdot v$ with $\mu(w)=0$.
			\item The vector $v$ is semistable if and only if there exists $0\ne w \in \overline{G\cdot v}$ with $\mu(w)=0$.
		\end{enumerate}
	\end{theorem}

	In the next section, the Laplace group model is defined. The maximum likelihood estimation problem of the multivariate symmetric Laplace distribution is explained and demonstrated, how it relates to the norm minimization over a group. And, a connection between the properties of the likelihood function of the multivariate symmetric Laplace distribution and the stability of data with respect to the group action, is established. 
	
	\section{Multivariate symmetric Laplace distribution}
	The density function of a $p$-dimensional symmetric Laplace distributed random vector $ \bm{Y}= (y_{1},y_{2},\cdots,y_{p})^{\top}$, $\bm{Y} \in \mathbb{R}^p$ with location parameter zero and scale parameter $\bm{\Sigma}_ {p\times p}$ (positive definite matrix), is
	\begin{equation}
		\label{eq:1}
		f_{\bm{Y}}(\bm{y}) = \frac{2}{(2\pi)^{\frac{p}{2}} \begin{vmatrix} \bm{\Sigma} \end{vmatrix}^{\frac{1}{2}}} \left( \frac{\bm{y}^\top \bm{\Sigma}^{-1} \bm{y}}{2} \right) ^{\nu/2}  K_{\nu} \left( \sqrt{2\bm{y}^\top \bm{\Sigma}^{-1} \bm{y}} \right),
	\end{equation}
	here, $\nu = \frac{2-p}{2}$, and $ K_{\nu}$ is the modified Bessel function of the third kind. This distribution is denoted as $\bm{Y} \sim \mathcal{SL}_{p}(\bm{\Sigma})$, see \cite{KKP}. For the definition and properties of the modified Bessel function of the third kind see \cite{B}, \cite{OLBC}, \cite{W}. 
	
	Next, a representation of the multivariate symmetric Laplace distribution is presented from our earlier work \cite{YS}, as this work is in process of publication, we reproduce the complete result here again.
	\begin{theorem}\cite{YS}
		\label{thm:5.1}
		A multivariate symmetric Laplace random variable $\bm{Y}$ has the representation 
		\begin{equation}
			\label{eq:5.2} 		
			\bm{Y}=\sqrt{W}\bm{Z},
		\end{equation}
		with random variable $\bm{Z}\sim \mathcal{N}_{p}(\bm{0}, \bm{\Sigma})$, the $p$-dimensional normal distribution with location parameter zero and scale parameter $\bm{\Sigma}$ and random variable $W$, independent of $\bm{Z}$, having a univariate exponential distribution with location parameter zero and scale parameter one.
	\end{theorem}
	\begin{proof}
		Since $W$ and $\bm{Z}$ are independent and $\bm{Y}=\sqrt{W}\bm{Z}$, then, the joint probability density function of $\bm{Y}$ and $W$ is
		\begin{align*}
			f_{\bm{Y},W} (\bm{y},w)&= f_{\bm{Z},W}(\bm{z},w) \left| J \right| \\
			&= f_{\bm{Z}}(\bm{z}) f_{W}(w)\left|J\right|,
		\end{align*}
		where $J = \begin{vmatrix} \frac{\partial(\bm{Z},W)}{\partial(\bm{Y},W)}	\end{vmatrix} =\frac{1}{(W)^{\frac{p}{2}}}$ is the Jacobian. Hence, the joint density function of $\bm{Y}$ and $W$ is
		
		\begin{equation}
			\label{eq:7.1}
			f_{\bm{Y},W}(\bm{y},w)= \frac{1}{(2\pi)^{\frac{p}{2}} \begin{vmatrix} \bm{\Sigma} \end{vmatrix}^{\frac{1}{2}}  w^\frac{p}{2}} \exp\left(-w-\frac{1}{2w} \bm{y}^\top  \bm{\Sigma}^{-1} \bm{y}\right),
		\end{equation}
		and the density function of $\bm{Y}$ is
		
		\begin{align*}
			f_{\bm{Y}}(\bm{y}) &= \int_{0}^{\infty} \frac{1}{(2\pi)^{\frac{p}{2}} \begin{vmatrix} \bm{\Sigma} \end{vmatrix}^{\frac{1}{2}}  w^\frac{p}{2}} \exp\left(-w-\frac{1}{2w} \bm{y}^\top  \bm{\Sigma}^{-1} \bm{y}\right) \,dw\\
			&=\frac{1}{(2\pi)^{\frac{p}{2}} \begin{vmatrix} \bm{\Sigma} \end{vmatrix}^{\frac{1}{2}} } \int_{0}^{\infty} \frac{1}{\left(w^{\frac{p-2}{2}+1}\right)}  \exp\left(-w-\frac{\left(\sqrt{2\bm{y}^\top  \bm{\Sigma}^{-1} \bm{y}}\right)^{2}}{4w} \right) \,dw \\
			&= \frac{2}{(2\pi)^{\frac{p}{2}} \begin{vmatrix} \bm{\Sigma} \end{vmatrix}^{\frac{1}{2}} } \left(\frac{\bm{y}^\top  \bm{\Sigma}^{-1} \bm{y}}{2}\right)^{\nu/2}  K_{-\nu}\left(\sqrt{2\bm{y}^\top  \bm{\Sigma}^{-1} \bm{y}}\right),
		\end{align*}
		where $\nu=\frac{2-p}{2}$, and $K_{-\nu}(x)$ is the modified Bessel function of the third kind given as
		
		\begin{equation*}
			K_{-\nu}(x)=\frac{1}{2}\left(\frac{x}{2}\right)^{-\nu} \int_{0}^{\infty} \frac{1}{(t)^{-\nu+1}} \exp\left(-t-\frac{x^2}{4t}\right) \,dt .
		\end{equation*}
		
		From the properties of modified Bessel function of the third kind, $K_{-\nu}(x)=K_{\nu}(x)$. Hence, the density function of $\bm{Y}$ is
		\begin{equation*}
			f_{\bm{Y}}(\bm{y})=\frac{2}{(2\pi)^{\frac{p}{2}} \begin{vmatrix} \bm{\Sigma} \end{vmatrix}^{\frac{1}{2}} } \left(\frac{\bm{y}^\top  \bm{\Sigma}^{-1} \bm{y}}{2}\right)^{\nu/2}  K_{\nu}\left(\sqrt{2\bm{y}^\top  \bm{\Sigma}^{-1} \bm{y}}\right).
		\end{equation*}
	\end{proof}

	\subsection{Multivariate symmetric Laplace distributions as group model}
	The Laplace model is a subset $\mathcal{L} \subseteq PD_{p}(\mathbb{R})$ which contains the concentration matrices $\bm{\Psi}=\bm{\Sigma}^{-1}$ of the multivariate $p$-dimensional symmetric Laplace distributions, i.e.
	\[\mathcal{L} = \{\mathcal{SL}_{p}(\bm{\Sigma}) \mid \bm{\Psi}=\bm{\Sigma}^{-1} \in PD_{p}(\mathbb{R})\}.\]
	
	\begin{definition}
		The Laplace model via symmterization is defined as 
		\[\mathcal{L}_{\mathcal{A}}= \{A^\top A \mid A\in \mathcal{A}\}, \] 
		for a subset $\mathcal{A}\subseteq GL_{p}(\mathbb{R})$. If $\mathcal{A}$ is subgroup of $GL_{p}(\mathbb{R})$, then it is called Laplace group model.
	\end{definition}

	\begin{theorem}
		Let $\mathcal{L} \subseteq PD_{p}(\mathbb{R})$ be a Laplace model, then there exist a subset $\mathcal{A}\subseteq GL_{p}(\mathbb{R})$ such that $\mathcal{L}=\mathcal{L}_{\mathcal{A}}$.
	\end{theorem}
	
	\begin{proof}
		A positive definite matrix $\bm{\Psi} \in PD_{p}(\mathbb{R})$ has a spectral decomposition, that is, $\bm{\Psi} =Q\Lambda Q^\top$, where $Q$ is orthogonal matrix and $\Lambda=\operatorname{diag}(\lambda_{1},\lambda_{2},\cdots, \lambda_{p})$, $\lambda_{i}>0$, then $\Lambda^{1/2}=\operatorname{diag}\left( \sqrt{\lambda_{1}},\sqrt{\lambda_{2}},\cdots, \sqrt{\lambda_{p}} \right)$. Consider $A=\bm{\Psi}^{1/2}=Q\Lambda^{1/2} Q^\top$, then $\bm{\Psi}=A^\top A$. Consider the subset $\mathcal{A}= \{\bm{\Psi}^{1/2} \mid \bm{\Psi} \in \mathcal{L}\} \subseteq GL_{p}(\mathbb{R})$. Then, the subset $\mathcal{A}$ satisfies $\mathcal{L}=\mathcal{L}_{\mathcal{A}}$. So, any Laplace model is of the form $\mathcal{L}_{\mathcal{A}}$.
	\end{proof}
	
	The Laplace group model is constructed using the representations of a group $G$ on a vector space $ \mathbb{R}^p$ ($G \to GL( \mathbb{R}^p)$). This construction depends on the image of the group $G$ in $GL( \mathbb{R}^p)$, so the group $G$ can be viewed as a subgroup of $GL( \mathbb{R}^p)$.
	
	\begin{definition}[Laplace group model]
		\label{def:7.1}
		The Laplace group model defined by $G\subseteq GL( \mathbb{R}^p)$ consists of multivariate symmetric Laplace distributions with $\bm{\Psi} \in PD_{p}(\mathbb{R})$ lying in the set 
		\[\mathcal{L}_{G}= \{A^\top A \mid A\in G\}.\]
	\end{definition}

	Next, the maximization of the log-likelihood function is compared to the minimization of the norm through the group actions.
	
	\subsection{Maximum likelihood estimation}
	Let $\underline{\bm{Y}}=(\bm{Y}_1,\bm{Y}_2,\ldots ,\bm{Y}_N)$ be random sample from a $p$-dimensional symmetric Laplace distribution $\mathcal{SL}_{p}(\bm{\Sigma})$. Then, the likelihood function is
	\begin{equation*}
		L(\bm{\Sigma})=f(\underline{\bm{Y}};\bm{\Sigma})= \prod_{i=1}^{N} f(\bm{Y_i}),
	\end{equation*}
	where $f(\bm{Y}_{i})$ is as in \eqref{eq:1}, and the log-likelihood function (up to an additive constant) is 
	\begin{equation}
		\label{eq:2}
		\ell(\bm{\Sigma})=- \frac{N}{2} \log \begin{vmatrix}\bm{\Sigma} \end{vmatrix} + \frac{\nu}{2} \sum_{i=1}^{N} \log ({\bm{Y}_{i}}^\top \bm{\Sigma}^{-1} \bm{Y}_{i})+ \sum_{i=1}^{N}\log K_{\nu} \left( \sqrt{2 ({\bm{Y}_{i}}^\top \bm{\Sigma}^{-1} \bm{Y}_{i})}  \right).
	\end{equation}
	A maximum likelihood estimate of the parameter $\bm{\Sigma}$ is that value of $\bm{\Sigma}\in PD_{p}(\mathbb{R})$, which maximizes the log-likelihood function \eqref{eq:2} given the data $\bm{Y}_{1},\bm{Y}_{2},\cdots,\bm{Y}_{N}$. 
	
	It is not easy to directly maximize the log-likelihood function \eqref{eq:2} with respect to the parameter $\bm{\Sigma}$, as it contains the modified Bessel function of the third kind. So, to overcome this problem, we are using \hyperref[thm:5.1]{Theorem \ref{thm:5.1}}, which is a representation of a multivariate symmetric Laplace random variable $\bm{Y}$. We are using the joint density function of $\bm{Y}$ and $W$ in the representation $\bm{Y}=\sqrt{W}\bm{Z}$ of a multivariate symmetric Laplace random variable $\bm{Y}$, where random variables $\bm{Z}\sim \mathcal{N}_{p}(0, \bm{\Sigma})$ and $W\sim Exp(1)$. 
	
	\begin{lemma}
		\label{lemma:5.2}
		Let $\left(\underline{\bm{Y}},\underline{W}\right)=\left(\bm{Y}_{1},\bm{Y}_{2},\cdots,\bm{Y}_{N},W_{1},W_{2},\cdots,W_{N}\right)$ be complete random sample from the joint distribution of $\bm{Y}$ and $W$, where $\bm{Y} \sim \mathcal{SL}_{p}(\bm{\Sigma})$ and $W\sim Exp(1)$. If $\bm{\Sigma}^{+}$ maximizes the complete data likelihood function, then $\bm{\Sigma}^{+}$ also maximizes the likelihood function of $p$-dimensional symmetric Laplace distribution.
	\end{lemma}
	\begin{proof}
		Let $\left(\underline{\bm{Y}},\underline{W}\right)$ be random sample from the joint distribution of $\bm{Y}$ and $W$, where $\bm{Y} \sim \mathcal{SL}_{p}(\bm{\Sigma})$ and $W\sim Exp(1)$. The complete data likelihood function for $\bm{\Sigma}$ is $L_{c}\left( \bm{\Sigma} \mid \underline{\bm{Y}},\underline{W} \right) =  f_{c}\left( (\underline{\bm{Y}},\underline{W});\bm{\Sigma} \right)$
		
		The likelihood function of $\bm{\Sigma}$ given the data $\underline{\bm{Y}}$, is
		\begin{equation*}
			L(\bm{\Sigma} \mid \underline{\bm{Y}})=f(\underline{\bm{Y}}; \bm{\Sigma})= \int_{\mathbb{R}^N} f_{c}\left( (\underline{\bm{Y}},\underline{W});\bm{\Sigma} \right)\, d\underline{W}=\int_{\mathbb{R}^N} L_{c}(\bm{\Sigma} \mid \underline{\bm{Y}},\underline{W} )\, d\underline{W}.
		\end{equation*}
		
		Now suppose
		\begin{equation*}
			\arg \max_{\bm{\Sigma} \in PD_{p}(\mathbb{R})} 	L_{c}\left( \bm{\Sigma} \mid \underline{\bm{Y}},\underline{W} \right)=\bm{\Sigma}^{+},
		\end{equation*}
		that is, 
		\begin{equation*}
			L_{c}\left( \bm{\Sigma} \mid \underline{\bm{Y}},\underline{W} \right)\le L_{c}\left( \bm{\Sigma}^{+} \mid \underline{\bm{Y}},\underline{W} \right)   \hspace{.9mm} \forall \bm{\Sigma}.
		\end{equation*}
		
		Thus,
		\begin{align*}
			L(\bm{\Sigma} \mid \underline{\bm{Y}} )&=\int_{\mathbb{R}^N} 	L_{c}\left( \bm{\Sigma} \mid \underline{\bm{Y}},\underline{W} \right)\, d\underline{W}\\
			& \le\int_{\mathbb{R}^N}  L_{c}( \bm{\Sigma}^{+} \mid \underline{\bm{Y}},\underline{W} ) \, d\underline{W} \hspace{.9mm} \forall \bm{\Sigma} \\
			&=L(\bm{\Sigma}^{+}\mid  \underline{\bm{Y}}),  \hspace{.9mm} \forall \bm{\Sigma}.
		\end{align*} 
		This implies 
		\[\arg \max_{\bm{\Sigma} \in PD_{p}(\mathbb{R})} L(\bm{\Sigma}\mid  \underline{\bm{Y}}) =\bm{\Sigma}^{+}.\]
	\end{proof}
	
	From the above \hyperref[lemma:5.2]{Lemma \ref{lemma:5.2}}, we can use the complete data likelihood function to find the MLE of $\bm{\Sigma}$. Using the joint probability density function of $\bm{Y}$ and $W$ given in equation \eqref{eq:7.1}, the complete data log-likelihood function (up to an additive constant) is 
	\begin{equation*}
		\ell_{c}(\bm{\Sigma})= -\frac{N}{2}\log \begin{vmatrix}\bm{\Sigma} \end{vmatrix} -\frac{1}{2}\sum_{i=1}^{N}\frac{1}{W_{i}} \left({\bm{Y}_{i}}^\top \bm{\Sigma}^{-1} \bm{Y}_{i}\right) -\sum_{i=1}^{N}\left(\frac{p}{2} \log W_{i} +W_{i}\right).
	\end{equation*}
	
	Since the last term of this equation does not contain any unknown parameter, it can be ignored for maximization of $\ell_{c}(\bm{\Sigma})$ with respect to $\bm{\Sigma}$.
	Therefore, for the MLE of the parameter $\bm{\Sigma}$, we have to maximize the complete data log-likelihood function (up to an additive constant)
	\begin{equation}
		\ell_{c}(\bm{\Sigma})=-\frac{N}{2}\log \begin{vmatrix}\bm{\Sigma} \end{vmatrix} -\frac{1}{2}\sum_{i=1}^{N}\frac{1}{W_{i}} \left({\bm{Y}_{i}}^\top \bm{\Sigma}^{-1} \bm{Y}_{i}\right),
	\end{equation}
	where $\left(\underline{\bm{Y}},\underline{W}\right)=\left(\bm{Y}_{1},\bm{Y}_{2},\cdots,\bm{Y}_{N},W_{1},W_{2},\cdots,W_{N}\right)$ is the random sample from the joint distribution of $\bm{Y}$ and $W$.

	\subsection{Maximizing log-likelihood as norm minimization} 
	In this subsection, the maximization of the complete data log-likelihood function is compared to the minimization of the norm through the group actions.
	
	For a random sample $(\bm{Y}_{1},\bm{Y}_{2},\cdots,\bm{Y}_{N},W_{1},W_{2},\cdots,W_{N})$ from the joint distribution of $\bm{Y}$ and $W$, let 
	\begin{equation}
		\label{eq:7.2}
		\bm{Y}^{W}=\left(\frac{1}{\sqrt{W_{1}}} \bm{Y}_{1},\frac{1}{\sqrt{W_{2}}}  \bm{Y}_{2},\cdots, \frac{1}{\sqrt{W_{N}}} \bm{Y}_{N}\right)\in  \mathbb{R}^{pN}.
	\end{equation}
	Then, the action of the group $G$ on $ \mathbb{R}^{pN}$ is defined $\forall A\in G$ and $\bm{Y}^{W}\in \mathbb{R}^{pN}$ as:	
	\[A. \bm{Y}^{W}=\left(\frac{1}{\sqrt{W_{1}}}A \bm{Y}_{1},\frac{1}{\sqrt{W_{2}}} A \bm{Y}_{2},\cdots,\frac{1}{\sqrt{W_{N}}}A \bm{Y}_{N}\right).\] 
	The action of $G$ on $ \mathbb{R}^{pN}$ is considered by taking the group $G$ as diagonally embedded in $GL( \mathbb{R}^{pN})$
	
	\begin{remark}
		\label{rem:7.1}
		If $\left(\underline{\bm{Y}},\underline{W}\right)=\left(\bm{Y}_{1},\bm{Y}_{2},\cdots,\bm{Y}_{N},W_{1},W_{2},\cdots,W_{N}\right)$ is the random sample from the joint distribution of $\bm{Y}$ and $W$, then \[ \bm{Y}^{W}=\left(\frac{1}{\sqrt{W_{1}}} \bm{Y}_{1},\frac{1}{\sqrt{W_{2}}}  \bm{Y}_{2},\cdots, \frac{1}{\sqrt{W_{N}}} \bm{Y}_{N}\right)\] is a statistic (data), and the MLE of the parameter given $\bm{Y}^{W}$ is equivalent to stating that the MLE of the parameter given $\left(\underline{\bm{Y}},\underline{W}\right)$. Here, we talk about the stability of the data $\bm{Y}^{W}$ rather than the random sample $\left(\underline{\bm{Y}},\underline{W}\right)$, with respect to the group action.
	\end{remark}
	
	The norm $\|A\cdot \bm{Y}^{W}\|$ can be written as 
	\begin{align}
		\label{eq:5}
		\|A\cdot  \bm{Y}^{W}\|^2 =\sum_{i=1}^{N} \left(\frac{1}{\sqrt{W_{i}}} A \bm{Y}_{i} \right)^\top \left(\frac{1}{\sqrt{W_{i}}} A \bm{Y}_{i} \right)
		=\sum_{i=1}^{N}\frac{1}{W_{i}} \operatorname{tr} \left(A^\top A \bm{Y}_{i} \bm{Y}_{i}^\top \right).
	\end{align}
	
	Now, consider the function to be maximized with respect to $\bm{\Sigma} \in PD_{p}(\mathbb{R})$
	\[	\ell_{c}(\bm{\Sigma})=-N \log \begin{vmatrix}\bm{\Sigma} \end{vmatrix} -\sum_{i=1}^{N} \frac{1}{W_{i}} { \bm{Y}_{i}}^\top \bm{\Sigma}^{-1}  \bm{Y}_{i} ,  \]
	Take  $\bm{\Sigma}^{-1} =\bm{\Psi} \in PD_{p}(\mathbb{R})$, then
	\begin{equation*} 
		\ell_{c}(\bm{\Psi})= N\log \begin{vmatrix}\bm{\Psi} \end{vmatrix} -\sum_{i=1}^{N}  \frac{1}{W_{i}} { \bm{Y}_{i}}^\top \bm{\Psi}  \bm{Y}_{i} =  N \log \begin{vmatrix}\bm{\Psi} \end{vmatrix} -\sum_{i=1}^{N}  \frac{1}{W_{i}} \operatorname{tr}\left(\bm{\Psi}  \bm{Y}_{i}{ \bm{Y}_{i}}^\top \right). 
	\end{equation*}  
	Since $\bm{\Psi} \in \mathcal{L}_{G}$, then $\bm{\Psi}$ can be written as $A^\top A$, where $A$ is an invertible matrix of order $p$ and 
	\begin{equation}
		\label{eq:6}
		\ell_{c}(\bm{\Psi})=	\ell_{c}(A^\top A)=N \log \begin{vmatrix} A^\top A \end{vmatrix} -\sum_{i=1}^{N}  \frac{1}{W_{i}} \operatorname{tr}(A^\top A  \bm{Y}_{i}{ \bm{Y}_{i}}^\top).
	\end{equation}
	
	Thus, comparing \eqref{eq:5} and \eqref{eq:6}, maximizing the complete data log-likelihood function $\ell_{c}(\bm{\Psi})$ over $\mathcal{L}_{G}=\{ \bm{\Psi}=A^\top A \mid A\in G\}$ is equivalent to minimizing
	\[ -\ell_{c}(A^\top A)=- N \log\begin{vmatrix} A^\top A \end{vmatrix}+\|A\cdot \bm{Y}^{W}\|^2 \]
	over $A\in G$.
	
	Next, we give the following result in which we can do this minimization into two steps, if the group is considered to be closed under non-zero scalar multiples. $G_{SL}^{\pm}$ is the subgroup of $G$ consisting of matrices with determinants $1$ or $-1$.
	
	\begin{theorem}
		\label{prop:2}
		Let $ \bm{Y}^{W}=\left(\frac{1}{\sqrt{W_{1}}} \bm{Y}_{1},\frac{1}{\sqrt{W_{2}}}  \bm{Y}_{2},\cdots, \frac{1}{\sqrt{W_{N}}} \bm{Y}_{N}\right) \in  \mathbb{R}^{pN}$ be as defined in \eqref{eq:7.2}. If the group $G\subseteq GL( \mathbb{R}^p)$ is closed under non-zero scalar multiples, then
		\begin{equation}
			\sup_{A\in G} \left(\ell_{c}(A^\top A) \right) \equiv
			\inf _{\alpha \in {\mathbb{R}}_{>0}} \left( \alpha \left( 	\inf_{B\in G^\pm_{SL}} \| B\cdot  \bm{Y}^{W}\|^2  \right)  -pN\log \alpha \right).
		\end{equation}
		The MLEs, if they exists are the matrices $\alpha B^\top B$, where $B$ minimizes the norm $\|B\cdot  \bm{Y}^{W}\|$ under the action of $G^\pm_{SL}$ on $\mathbb{R}^{pN}$ and $\alpha \in {\mathbb{R}}_{>0}$ is the unique value minimizing the outer infimum. 
	\end{theorem}

	\begin{proof}
		Maximizing $\ell_{c}(\bm{\Psi})$ over $\mathcal{L}_{G}$ is equivalent to minimizing 
		\[f(A)=-\ell_{c}(A^\top A)= \|A\cdot  \bm{Y}^{W}\|^2 - N \log\begin{vmatrix}A^\top A \end{vmatrix}, \]
		over $A\in G$.
		Since, the group $G\subseteq GL( \mathbb{R}^p)$ is closed under non-zero scalar multiples, so, $A\in G$ can be written as $A=\beta B$, where $\beta \in \mathbb{R}_{>0}$ and $B\in G^\pm_{SL}$. Using $A^\top A =\beta^2B^\top B$, and setting $\alpha=\beta ^2$,
		\begin{align*}
			f(A)&= \| \beta B\cdot  \bm{Y}^{W}\|^2 -  N \log\begin{vmatrix}\beta ^2B^\top B \end{vmatrix}\\
			&=\beta ^2 \|B\cdot  \bm{Y}^{W}\|^2 - N \log\left((\beta^2)^{p} \right)\\
			&=\alpha \|B\cdot  \bm{Y}^{W}\|^2 - pN \log\left(\alpha \right).
		\end{align*}
		
		Now, let $\|B\cdot  \bm{Y}^{W}\|^2=c$, then $f(A)$ becomes the function of single variable $\alpha $. The minimum value of the function $\alpha c-pN \log(\alpha)$ in $\alpha$ is 
		\[Np [(1-\log(Np)) +\log (c)] \] 
		for $c>0$, which increases as $c$ increases, so the minimum value is attained, when $c$ is minimum. Hence, to minimize $f$ with respect to $A\in G$, we first minimize the norm over $G_{SL}^\pm$ and then minimize the univariate function in $\alpha$, that is,
		\begin{equation}
			\inf_{A\in G}f(A)=\inf _{\alpha \in {\mathbb{R}}_{>0}} \left( \alpha \left( 	\inf_{B\in G^\pm_{SL}} \|B\cdot  \bm{Y}^{W}\|^2  \right)  -pN\log \alpha \right).
		\end{equation}
		Furthermore, an MLE is a matrix $\hat{\bm{\Psi}} \in \mathcal{L}_{G}$ that maximizes $\ell_{c}(\bm{\Psi})$. Hence, the MLEs are all the matrices $\bm{\hat{\Psi}}=A^\top A =\alpha B^\top B$, where $A=\sqrt{\alpha}B$, and $B$ and $\alpha$ minimize the inner and outer infima, respectively.
	\end{proof}

	\subsection{Correspondence between stability and maximum likelihood estimation} 
	In this subsection, the correspondence between the stability of data with respect to the group action and existence and uniqueness of the MLE is proved using \hyperref[prop:2]{Theorem \ref{prop:2}}.

	\begin{lemma}
		\label{lem:4}
		Let $G$ be a Zariski closed self-adjoint group, closed under non zero scalar multiples. If there is an element of $G$ with negative determinant, then $G$ contains an orthogonal element of determinant $-1$. Moreover, \hyperref[prop:2]{Theorem \ref{prop:2}} holds for $G^{+}_{SL}$ as well as $G^\pm_{SL}$.
	\end{lemma}
	
	\begin{proof}
		Let $A\in G$ with $\det(A)<0$. Since $G$ is Zariski closed and self-adjoint, then by the polar decomposition in $G$, any $A\in G$ can be written as $A=OP$, where $O\in G$ is an orthogonal matrix and $P\in G$ is positive definite matrix. Then, $\det(A)<0$ and $P\in G$ is positive definite implies that $\det(O)=-1$. 
		
		If there are no matrices in $G$ with negative determinant, then $G^\pm_{SL}=G^{+}_{SL}$, and if there is a matrix $A\in G$ of negative determinant then from above proof $G$ contains an orthogonal matrix $O \in G$ of determinant $-1$. Then, $A\in G$ can be written as $A=\beta OB$, where $\beta \in \mathbb{R}_{>0}$ and $B\in G^{+}_{SL}$, and follow the \hyperref[prop:2]{Theorem \ref{prop:2}} with $\bm{\Psi}=A^\top A=\beta ^{2} B^\top O^\top OB =\beta ^2B^\top B$, then minimizing $\|B\cdot  \bm{Y}^{W}\|$ over $G^\pm_{SL}$ is equivalent to minimizing $\|B\cdot  \bm{Y}^{W}\|$ over $G^{+}_{SL}$. Hence, \hyperref[prop:2]{Theorem \ref{prop:2}} holds for $G^{+}_{SL}$ as well as $G^\pm_{SL}$.
	\end{proof}	
	
	As per the result in \hyperref[lem:4]{Lemma \ref{lem:4}}, $G^{+}_{SL}$ will be used instead of $G^\pm_{SL}$ in further discussion.
	
	\begin{theorem}
		\label{prop:5}
		Let $ \bm{Y}^{W} \in  \mathbb{R}^{pN}$ be as defined in \eqref{eq:7.2}, and let $G\subseteq GL( \mathbb{R}^p)$ be a Zariski closed self-adjoint group which is closed under non-zero scalar multiplication. If $\alpha B^\top B$ is an MLE given $\left(\underline{\bm{Y}},\underline{W}\right)$, with $B\in G^{+}_{SL}$ and $\alpha \in \mathbb{R}_{>0}$, then all MLEs given $\left(\underline{\bm{Y}},\underline{W}\right)$ are of the form $S^\top (\alpha B^\top B)S$, where $S$ is in the $G^{+}_{SL}$-stabilizer of $ \bm{Y}^{W}$. 
	\end{theorem}

	\begin{proof}
		From  \hyperref[prop:2]{Theorem \ref{prop:2}}, $\alpha B^\top B$ is an MLE given $ \bm{Y}^{W}$, where $B\in G^{+}_{SL}$ and $\alpha \in \mathbb{R}_{>0}$ satisfies the equation 
		\begin{equation*}
			\inf _{\alpha \in {\mathbb{R}}_{>0}} \left( \alpha \left( 	\inf_{B\in G^{+}_{SL}} \| B\cdot  \bm{Y}^{W}\|^2  \right)  -pN\log \alpha \right),
		\end{equation*}
		that is, $B\in G^{+}_{SL}$ minimize the norm of $ \bm{Y}^{W}$ under the action of $G^{+}_{SL}$.
		
		Consider the $G^{+}_{SL}$-stabilizer of $ \bm{Y}^{W}$ $T_{\bm{Y}^{W}}= \{S\in G^{+}_{SL} \mid S\cdot  \bm{Y}^{W}= \bm{Y}^{W} \}$. Thus, $BS$ minimize the norm of $ \bm{Y}^{W}$ under the action of $G^{+}_{SL}$, for any $S$ is in the $G^{+}_{SL}$-stabilizer of $ \bm{Y}^{W}$, because
		\[ \inf_{\tilde{B}\in G^{+}_{SL}}\|\tilde{B} \cdot  \bm{Y}^{W}\|^2 =\|B\cdot  \bm{Y}^{W}\|^2 = \|B\cdot (S\cdot  \bm{Y}^{W})\|^2=\|(BS)\cdot  \bm{Y}^{W}\|^2. \]
		Therefore, $\alpha (BS)^\top BS=S^\top (\alpha B^\top B)S$ is also a MLE given $ \bm{Y}^{W}$. 
		
		Conversely, by \hyperref[prop:2]{Theorem \ref{prop:2}}, any MLE of the form $\alpha (B')^\top B'$ with $B' \in G^{+}_{SL} $ such that
		\[ \|B' \cdot  \bm{Y}^{W}\|^2 =\inf_{\tilde{B}\in G^{+}_{SL}}\|\tilde{B} \cdot  \bm{Y}^{W}\|^2 =\|B\cdot  \bm{Y}^{W}\|^2. \]
		Since, $B\cdot  \bm{Y}^{W}$ is of minimal norm in its orbit and $\|B' \cdot  \bm{Y}^{W}\|^2 =\|B\cdot  \bm{Y}^{W}\|^2 $, therefore by (a) and (b) of the Kempf-Ness \hyperref[thm:1]{Theorem \ref{thm:1}}, there exists an orthogonal matrix $O\in G^{+}_{SL}$ with $B'\cdot  \bm{Y} ^{W}= O\cdot(B\cdot  \bm{Y}^{W})$ implies that $S= B^{-1}O^{-1}B'$ is in the $G^{+}_{SL}$-stabilizer of $ \bm{Y}^{W}$, and $B'=OBS$, and deduce that the MLE $\alpha (B')^\top B'= S^\top (\alpha B^\top B)S$. 
	\end{proof}

	\begin{theorem}
		\label{thm:6 }
		Let $ \bm{Y}^{W} \in  \mathbb{R}^{pN}$ be as defined in \eqref{eq:7.2} and a group $G\subseteq GL(\mathbb{R}^p)$ be a Zariski closed self-adjoint group which is closed under non-zero scalar multiples. The stability under the action of $G^{+}_{SL}$ on $ \mathbb{R}^{pN}$ is related to maximum likelihood estimation of the multivariate symmetric Laplace distribution over $\mathcal{L}_{G}$ as follows:
		
		\renewcommand{\labelenumi}{(\alph{enumi})}
		\begin{enumerate}
			\item $ \bm{Y}^{W}$ is unstable $\iff$ $\ell_{c}(\bm{\Psi})$ is not bounded from above,
			\item $ \bm{Y}^{W}$ is semistable $\iff$ $\ell_{c}(\bm{\Psi})$ is bounded from above,
			\item $ \bm{Y}^{W}$ is polystable $\iff$ MLE exists,
			\item  $ \bm{Y}^{W}$ is stable $\implies$ finitely many MLEs exists $\iff$ unique MLE exists.
		\end{enumerate}
	\end{theorem}
	
	\begin{proof}
		If $\bm{Y}^{W}$ is unstable, then $c= \inf_{B\in G^{+}_{SL}} \|B\cdot  \bm{Y}^{W}\|^2 =0$. Hence, from \hyperref[prop:2]{Theorem \ref{prop:2}}, 
		\[ \inf_{A\in G}(f(A))= \inf _{\alpha \in {\mathbb{R}}_{>0}} \left(0 -pN\log \alpha \right), \]
		the outer infima tends to $-\infty$, so the maximum of $\ell_{c}(\bm{\Psi})$ tends to $\infty$. Thus, $\ell_{c}(\bm{\Psi})$ is not bounded from above. If $ \bm{Y}^{W}$ is not unstable, that is, $ \bm{Y}^{W}$ is semistable. Conversely, suppose that  $ \bm{Y}^{W}$ is semistable, then $c= \inf_{B\in G^{+}_{SL}} \|B\cdot  \bm{Y}^{W}\|^2 >0$ and thus, by \hyperref[prop:2]{Theorem \ref{prop:2}}, 
		\[ \inf_{A\in G}(f(A))= \inf _{\alpha \in {\mathbb{R}}_{>0}} \left(\alpha c -pN\log \alpha \right), \]
		the outer infima is some real number, and hence, $\ell_{c}(\bm{\Psi})$ is bounded from above. This proves parts (a) and (b). 
		
		If $ \bm{Y}^{W}$ is polystable, then $c= \inf_{B\in ^{+}_{SL}} \|B\cdot  \bm{Y}^{W}\|^2 >0$ is attained for some $B\in G^{+}_{SL}$ and $\alpha B^\top B$ is an MLE, where $\alpha \in \mathbb{R}_{>0}$ minimize the outer infimum. Conversely,	if an MLE given $ \bm{Y}^{W}$ exists, then $\ell_{c}(\bm{\Psi})$ is bounded from above and attains its maximum. Hence, the double infimum from \hyperref[prop:2]{Theorem \ref{prop:2}} is attained, and there exists $B\in G^{+}_{SL}$ such that $B\cdot  \bm{Y}^{W}$ has minimal norm in its orbit under $G^{+}_{SL}$, thus by the Kempf-Ness \hyperref[thm:1]{Theorem \ref{thm:1}} (a) and (d), the orbit $G\cdot \bm{Y}^{W}$ is closed. Hence, $ \bm{Y}^{W}$ is polystable.
		
		If $ \bm{Y}^{W}$ is stable, then its stabilizer $T_{ \bm{Y}^{W}}$ is finite, then by \hyperref[prop:5]{Theorem \ref{prop:5}}, there are only finitely many MLEs given $ \bm{Y}^{W}$. It remains to show that a tuple $ \bm{Y}^{W}$ can't have finitely many MLEs unless its unique.
		Since, there are finitely many MLEs given $ \bm{Y}^{W}$, so by part (c), $ \bm{Y}^{W}$ is polystable. Next, we relate the $T_{ \bm{Y}^{W}}$ to $T_{B\cdot \bm{Y}^{W}}$.
		\[T_{ \bm{Y}^{W}}=\{S\in G^{+}_{SL} \mid S\cdot  \bm{Y}^{W}= \bm{Y}^{W} \}\]
		\[T_{B\cdot  \bm{Y}^{W}}=\{S'\in G^{+}_{SL} \mid S'\cdot B\cdot  \bm{Y}^{W}=B\cdot  \bm{Y}^{W} \}\]
		
		We claim that $T_{ \bm{Y}^{W}}= B^{-1}T_{B\cdot  \bm{Y}^{W}} B$.
		
		Let $S_{1}\in T_{ \bm{Y}^{W}} \implies S_{1}\cdot \bm{Y}^{W}=\bm{Y}^{W}$.
		
		Now, $S_{1}\cdot (B^{-1}B\cdot \bm{Y}^{W})= B^{-1}B\cdot \bm{Y}^{W} \implies BS_{1}B^{-1}\cdot (B\cdot \bm{Y}^{W})= B\cdot \bm{Y}^{W} \implies  BS_{1}B^{-1} \in T_{B\cdot \bm{Y}^{W}}$. Hence, $S_{1}\in B^{-1} T_{\bm{Y}^{W}}B$, that is, 
		$T_{ \bm{Y}^{W}} \subseteq B^{-1}T_{B\cdot  \bm{Y}^{W}} B.$ 
		
		Now, let $S_{2} \in  B^{-1}T_{B\cdot  \bm{Y}^{W}} B \implies S_{2}= B^{-1}S_{3} B$ for some $S_{3}\in  T_{B\cdot  \bm{Y}^{W}} $, and $B S_{2}=S_{3} B\implies B S_{2} \cdot \bm{Y}^{W}=S_{3} B \cdot \bm{Y}^{W} \implies B S_{2} \cdot \bm{Y}^{W}=B \cdot \bm{Y}^{W}\implies S_{2} \cdot \bm{Y}^{W}=  \bm{Y}^{W}$. Hence, $S_{2} \in T_{\bm{Y}^{W}}$, that is, $B^{-1}T_{B\cdot \bm{Y}^{W}} B \subseteq T_{\bm{Y}^{W}}$. Thus, $T_{ \bm{Y}^{W}}= B^{-1} T_{B\cdot  \bm{Y}^{W}} B$.
		
		Let $\alpha A^\top A$ is an MLE given $\bm{Y}^{W}$, then by \hyperref[prop:5]{Theorem \ref{prop:5}}, all MLEs given $\bm{Y}^{W}$ are ${S_{1}}^\top \alpha A^\top A S_{1}$ with $S_{1}\in T_{\bm{Y}^{W}}.$ Now, using the equality $T_{ \bm{Y}^{W}}= B^{-1}T_{B\cdot  \bm{Y}^{W}} B$, since $S_{1}\in T_{\bm{Y}^{W}}$, then there exist $S_{2}\in T_{B \cdot \bm{Y}^{W}}$ such that $S_{1}=B^{-1}S_{2}B$.
		Now all the MLEs given $\bm{Y}^{W}$ are 
		\begin{align*}
			{S_{1}}^\top \alpha A^\top A S_{1} &= \left(B^{-1}S_{2} B \right)^\top  \left( \alpha A^\top A\right) \left(B^{-1}S_{2}B\right) \\
			&= B^\top {S_{2}}^\top (B^{-1})^\top  \left( \alpha A^\top A\right) \left(B^{-1}S_{2}B\right) \\
			&= B^\top {S_{2}}^\top \left( \alpha \left(AB^{-1}\right)^\top  \left(AB^{-1} \right) \right) S_{2} B.
		\end{align*}
		Since, $S_{2}\in T_{B \cdot \bm{Y}^{W}} \implies {S_{2}}^\top \left( \alpha \left(AB^{-1}\right)^\top (AB^{-1})\right) S_{2}$ are all the MLEs given $B\cdot \bm{Y}^{W}$ (because if $\alpha A^\top A$ is an MLE given $\bm{Y}^{W}$, then $\alpha \left(AB^{-1}\right)^\top  \left(AB^{-1}\right)$ is an MLE given $B\cdot \bm{Y}^{W}$). Hence,                                                               
		\[\{\text{MLEs given}\hspace{.7mm} \bm{Y}^{W}\}={B}^\top \{\text{MLEs given}\hspace{.7mm} B\cdot   \bm{Y}^{W}\}B.\]                                                                                                                                                        	                                                                                                                                                                     
		To study the stabilizer and MLE of polystable $ \bm{Y}^{W}$, without loss of generality we can assume that $ \bm{Y}^{W}$ is of minimal norm in its orbit. So, one of the MLEs given $ \bm{Y}^{W}$ is $\alpha  \bm{I}$, where $\alpha \in \mathbb{R}_{>0}$ minimize the outer infimum in \hyperref[prop:2]{Theorem \ref{prop:2}} and $ \bm{I}$ is the identity matrix of order $p$.
		
		Next, we show that the set $\{S^\top S\mid S\in T_{ \bm{Y}^{W}}\}$ is either singleton having only identity matrix or an infinite set, that is, $ \bm{Y}^{W}$ has either a unique MLE or infinitely many MLEs because the MLEs given $ \bm{Y}^{W}$ are the matrices $S^\top (\alpha \bm{I}^\top  \bm{I})S=\alpha S^\top S$, where $S\in T_{ \bm{Y}^{W}}$. The subgroup $T_{ \bm{Y}^{W}}$ is self-adjoint. If $T_{ \bm{Y}^{W}}$ is contained into the set of orthogonal matrices, then $\{S^\top S\mid S\in T_{ \bm{Y}^{W}}\}=\{ \bm{I}\}$. If not, that is, there is a matrix $K\in T_{ \bm{Y}^{W}}$ such that $K$ is non-orthogonal matrix. Since, $ T_{ \bm{Y}^{W}}$ is a self-adjoint group, $K^\top \in T_{ \bm{Y}^{W}}$ and $K^\top K\in  T_{ \bm{Y}^{W}}$, and $K^\top K \neq  \bm{I}$, so $K^\top K$ has at least one eigenvalue $\lambda_{1}$ other than one, that is, $\lambda_{1} \ne 1$. Since $T_{\bm{Y}^{W}}$ is a subgroup, so $(K)^{n}, \left(K^\top\right)^{n} \in T_{\bm{Y}^{W}}$, and $\left(K^\top K\right)^n \in  T_{ \bm{Y}^{W}}$ for some positive integer $n$, and at least one eigenvalue of $\left(K^\top K\right)^n$ is ${\lambda_{1}}^n \ne 1$. Thus, the set $\{ \left( K^{n}\right)^\top (K)^n= \left(K^\top K\right)^n \mid  (K)^n  \in  T_{ \bm{Y}^{W}} \}\subseteq \{S^\top S\mid S\in T_{ \bm{Y}^{W}}\}$ is infinite. Hence, $\{S^\top S\mid S\in T_{ \bm{Y}^{W}}\}$ is infinite.
	\end{proof}
	
	In the next example, the importance of the assumption that the group is self adjoint is addressed, that is, the condition (d) of \hyperref[thm:6 ]{Theorem \ref{thm:6 }}, which says that, finitely many MLEs exists $\implies$ unique MLE exist, is not true if the group $G$ is not self adjoint .
	
	\begin{example}
		Let the group $G$ is
		\[G= \bigcup_{\lambda \in \bm{R}\setminus \{0\}} \{ \lambda \bm{I}_{2}, \lambda \bm{S}\},\]
		where $\bm{S}=\begin{bmatrix}
			1 &-1\\0 &-1
		\end{bmatrix}$ and $\bm{I}_{2}$ is an identity matrix of order $2$.
		
		Consider the random sample $\left(\underline{\bm{Y}}, \underline{W}\right) =\left( \bm{Y}_{1}, W_{1} \right)=\left( \begin{pmatrix}
			2\\0
		\end{pmatrix},1 \right)$, then $\bm{Y}^{W}=\left( \frac{1}{\sqrt{W_{1}}} \bm{Y}_{1}\right)= \begin{pmatrix}
			2\\0
		\end{pmatrix} \in \mathbb{R}^{2}$. The subgroup $G^\pm_{SL}=\{ \pm \bm{I}_{2}, \pm \bm{S}\}$, and the MLEs given $\bm{Y}^{W}$ are $\alpha B^\top B$, where $\alpha >0$ is the unique value which minimizes the outer infimum and $B\in G^\pm_{SL}$ minimize the norm $\|A\cdot \bm{Y}^{W}\|$, given in \hyperref[prop:2]{Theorem \ref{prop:2}}. 
		
		Now, $\|\bm{Y}^{W}\|^{2} =\|\bm{S} \cdot \bm{Y}^{W}\|^{2}$, and all elements in the subgroup $G^\pm_{SL}$ minimizes the norm $\|A\cdot \bm{Y}^{W}\|$. Therefore, there are exactly two distinct MLEs $\alpha {\bm{I}_{2}}^\top \bm{I}_{2}$ and $\alpha {\bm{S}}^\top \bm{S}$ given the data $\bm{Y}^{W}$.
	\end{example} 
	
	\subsection{Complex Laplace group model} The complex Laplace group model is constructed using the representations $G \to GL( \mathbb{C}^p)$ of a group $G$ on a vector space $ \mathbb{C}^p$.  This construction depends on the image of the group $G$ in $GL( \mathbb{C}^p)$, so the group $G$ can be viewed as a subgroup of $GL( \mathbb{C}^p)$. The Laplace group model given by $G$ is the set of multivariate symmetric Laplace distributions with $\bm{\Psi}=\bm{\Sigma}^{-1}\in PD_{p}(\mathbb{R})$, and $\bm{\Psi} \in \mathcal{L}_{G}$, where
	\[\mathcal{L}_{G}= \{A^* A \mid A\in G\}.\]
	Then, the log-likelihood function (up to an additive constant) becomes
	\begin{equation*}
		\ell_{c}(\bm{\Psi})= N \log \begin{vmatrix}\bm{\Psi} \end{vmatrix} -\sum_{i=1}^{N}  \frac{1}{|W_{i}|} \operatorname{tr}\left(\bm{\Psi}  \bm{Y}_{i}{ \bm{Y}_{i}}^{*} \right). 
	\end{equation*}  
	
	For a random sample $(\bm{Y}_{1},\bm{Y}_{2},\cdots,\bm{Y}_{N},W_{1},W_{2},\cdots,W_{N})$ from the joint distribution of $\bm{Y}$ and $W$, let 
	\begin{equation}
		\label{eq:7.3}
		\bm{Y}^{W}=\left(\frac{1}{\sqrt{W_{1}}} \bm{Y}_{1},\frac{1}{\sqrt{W_{2}}}  \bm{Y}_{2},\cdots, \frac{1}{\sqrt{W_{N}}} \bm{Y}_{N}\right)\in  \mathbb{C}^{pN}.
	\end{equation}
	Then, the action of the group $G\subseteq GL(\mathbb{C}^p)$ on $ \mathbb{C}^{pN}$ is defined $\forall A\in G$ and $\bm{Y}^{W}\in \mathbb{R}^{pN}$ as:	
	\[A. \bm{Y}^{W}=\left(\frac{1}{\sqrt{W_{1}}}A\bm{Y}_{1},\frac{1}{\sqrt{W_{2}}}A \bm{Y}_{2},\cdots,\frac{1}{\sqrt{W_{N}}} A\bm{Y}_{N}\right).\]
	(If $G$ acts on $ \mathbb{C}^{pN}$, then the group $G$ is diagonally embedded in $GL( \mathbb{C}^{pN})$ ). The \hyperref[rem:7.1]{Remark \ref{rem:7.1}} applies here as well.
	
	The norm can be written as 
	\begin{equation}
		\|A\cdot  \bm{Y}^{W}\|^2 =\sum_{i=1}^{N} \left(\frac{1}{\sqrt{W_{i}}} A \bm{Y}_{i} \right)^{*} \left(\frac{1}{\sqrt{W_{i}}} A \bm{Y}_{i} \right)	=\sum_{i=1}^{N}\frac{1}{|W_{i}|} \operatorname{tr} \left(A^{*} A  \bm{Y}_{i} \bm{Y}_{i}^{*} \right).
	\end{equation}
	
	Hence, maximizing the log-likelihood over $\mathcal{L}_{G}=\{ \bm{\Psi}=A^{*} A \mid A\in G\}$ is equivalent to minimizing
	\[ -\ell_{c}(\bm{\Psi})=-\ell_{c}(A^{*} A)=- N \log\begin{vmatrix} A^{*}A \end{vmatrix}+\|A\cdot  \bm{Y}^{W}\|^2 \]
	over $A\in G$.
	
	The minimization can be done into two steps as given in \hyperref[prop:2]{Theorem \ref{prop:2}}, but here we are working over complex field, so the norm minimization can be done over the subgroup $G^{+}_{SL} \subseteq G$ of matrices with determinant one, instead of $G^\pm_{SL}$, without the extra condition on $G$, that is, $G$ contains an orthogonal matrix of determinant $-1$, given in \hyperref[lem:4]{Lemma \ref{lem:4}}.
	
	\begin{theorem}
		\label{prop:8}
		Let $ \bm{Y}^{W}=\left(\frac{1}{\sqrt{W_{1}}} \bm{Y}_{1},\frac{1}{\sqrt{W_{2}}}  \bm{Y}_{2},\cdots, \frac{1}{\sqrt{W_{N}}} \bm{Y}_{N}\right) \in  \mathbb{C}^{pN}$ be as defined in \eqref{eq:7.3}. If the group $G\subseteq GL( \mathbb{C}^p)$ is closed under non-zero complex scalar multiplication, then
		\begin{equation}
			\sup_{A\in G} \left(\ell_{c}(A^{*} A) \right) \equiv
			\inf _{\alpha\in {\mathbb{R}}_{>0}} \left( \alpha \left( \inf_{B\in  G^{+}_{SL}} \| B\cdot  \bm{Y}^{W}\|^2  \right)  -pN\log \alpha \right).
		\end{equation}
		The MLEs, if they exists are the matrices $\alpha B^* B$, where $B$ minimizes the norm $\|B\cdot \bm{Y}^{W}\|$ under the action of $G^{+}_{SL}$ on $ \mathbb{C}^{pN}$ and $\alpha \in {\mathbb{R}}_{>0}$ is the unique value minimizing the outer infimum. 
	\end{theorem}
	
	\begin{proof}
		Maximizing $\ell_{c}(\bm{\Psi})$ over $\mathcal{L}_{G}$ is equivalent to minimizing 
		\[f(A)= \|A\cdot  \bm{Y}^{W}\|^2 - N \log\begin{vmatrix} A^{*} A \end{vmatrix}, \]
		over $A\in G$.
		Since, the group $G$ is closed under non-zero complex scalar multiplication, so, $A\in G$ can be written as $A=\beta B$, where $\beta \in \mathbb{C}\setminus \{0\}$ and $B\in G^{+}_{SL}$. Using $A^{*} A= |\beta|^2 B^{*} B$, and setting $\alpha=|\beta|^2$,
		\begin{align*}
			f(A)&= \| \beta B\cdot  \bm{Y}^{W}\|^2 -  N \log\begin{vmatrix} |\beta| ^2B^{*} B \end{vmatrix}\\
			&=|\beta| ^2 \|B\cdot  \bm{Y}^{W}\|^2 - N\log\left((|\beta|^2)^{p} \right)\\
			&=\alpha \|B\cdot  \bm{Y}^{W}\|^2 - pN \log\left(\alpha \right)
		\end{align*}
		
		Now, following the proof of \hyperref[prop:2]{Theorem \ref{prop:2}}, 
		\begin{equation}
			\inf_{A\in G}f(A)=\inf _{\alpha \in {\mathbb{R}}_{>0}} \left( \alpha \left(	\inf_{B\in G^{+}_{SL}} \|B\cdot  \bm{Y}^{W}\|^2  \right)  -pN\log \alpha \right).
		\end{equation}
		Hence, the MLEs are all the matrices $\bm{\hat{\Psi}}=A^{*} A=\alpha B^{*}B$, where $B$ and $\alpha$ minimize the inner and outer infima, respectively.
	\end{proof}
	
	The difference between working with real Laplace models and complex Laplace models is that the converse part of \hyperref[thm:6 ]{Theorem \ref{thm:6 }} (d) is true over $\mathbb{C}$, which is not true over $\mathbb{R}$. First the analogue of \hyperref[prop:5]{Theorem \ref{prop:5}} is given for complex Laplace models.
	
	\begin{theorem}
		\label{prop:9}
		Let $ \bm{Y}^{W} \in  \mathbb{C}^{pN}$ be as defined in \eqref{eq:7.3}, and let $G\subseteq GL( \mathbb{C}^p)$ be a Zariski closed self-adjoint group which is closed under non-zero complex scalar multiplication. If $\alpha B^{*} B$ is an MLE given $ \left(\underline{\bm{Y}},\underline{W}\right)$, with $B\in G^{+}_{SL}$ and $\alpha \in \mathbb{R}_{>0}$, then all MLEs given $\left(\underline{\bm{Y}},\underline{W}\right)$ are of the form $S^{*}(\alpha B^{*} B)S$, where $S$ is in the $G^{+}_{SL}$-stabilizer of $ \bm{Y}^{W}$. 
	\end{theorem}	
	
	\begin{proof}
		The proof of this Theorem is similar to the proof of \hyperref[prop:5]{Theorem \ref{prop:5}} using the \hyperref[prop:8]{Theorem \ref{prop:8}} and the Kempf-Ness \hyperref[thm:1]{Theorem \ref{thm:1}} over $\mathbb{C}$.
	\end{proof}
	
	\begin{theorem}
		\label{thm:10}
		Let $ \bm{Y}^{W} \in  \mathbb{C}^{pN}$ be as defined in \eqref{eq:7.3} and a group $G\subseteq GL( \mathbb{C}^p)$ be a Zariski closed self-adjoint group which is closed under non-zero complex scalar multiplication. The stability under the action of $G^{+}_{SL}$ on $\mathbb{C}^{pN}$ is related to the maximum likelihood estimation of the multivariate symmetric Laplace distribution over $\mathcal{L}_{G}$ as follows:
		
		\renewcommand{\labelenumi}{(\alph{enumi})}
		\begin{enumerate}
			\item $ \bm{Y}^{W}$ is unstable $\iff$ $\ell_{c}(\bm{\Psi})$ is not bounded from above,
			\item $ \bm{Y}^{W}$ is semistable $\iff$ $\ell_{c}(\bm{\Psi})$ is bounded from above,
			\item $ \bm{Y}^{W}$ is polystable $\iff$ MLE exists,
			\item  $ \bm{Y}^{W}$ is stable $\iff $ finitely many MLEs exists $\iff$ unique MLE exists.
		\end{enumerate}
	\end{theorem}
	
	\begin{proof}
		The only converse part of (d) is remain to prove as the proofs of the other parts are similar as done for the real Laplace group models.
		
		Assume that the MLE given $ \bm{Y}^{W}$ is unique, so from part (c), $ \bm{Y}^{W}$ is polystable. To prove that $ \bm{Y}^{W}$ is stable, we need to show that $G^{+}_{SL}$-stabilizer of $ \bm{Y}^{W}$, that is, $T_{ \bm{Y}^{W}}$ is finite. For the complex setting, using the similar argument as in the proof of \hyperref[thm:6 ]{Theorem \ref{thm:6 }}, we claim that $T_{ \bm{Y}^{W}}= B^{-1}T_{B\cdot  \bm{Y}^{W}} B$, and from \hyperref[prop:8]{Theorem \ref{prop:8}},
		\[\{\text{MLEs given}\hspace{.7mm} \bm{Y}^{W}\}={B}^{*}\{\text{MLEs given}\hspace{.7mm}  B\cdot  \bm{Y}^{W}\}B.\] 
		
		Using the above property of $T_{ \bm{Y}^{W}}$, assume that $\bm{Y}^{W}$ is of minimal norm in its orbit. So, one of the MLEs given $ \bm{Y}^{W}$ is $\alpha \bm{I}$, $\alpha \in \mathbb{R}_{>0}$ minimize the outer infimum in \hyperref[prop:8]{Theorem \ref{prop:8}} and $ \bm{I}$ is the identity matrix of order $p$. By \hyperref[prop:9]{Theorem \ref{prop:9}}, all the MLEs given $\bm{Y}^{W}$ are of the form $S^{*}(\alpha \bm{I}^{*} \bm{I} )S=\alpha S^{*}S$, where $S\in T_{ \bm{Y}^{W}}$. Since, $\alpha \bm{I}$ is the unique MLE, so, $S^{*}S = \bm{I}$, where $S\in T_{ \bm{Y}^{W}}$. Thus, $T_{\bm{Y}^{W}}$ is contained into the group of unitary matrices in $G$. Therefore, $T_{ \bm{Y}^{W}}$ is $\mathbb{C}$-compact. The subgroup $T_{ \bm{Y}^{W}}$ is Zariski closed (as it is defined by the polynomial equation $S\cdot  \bm{Y}^{W}= \bm{Y}^{W}$). Since $T_{ \bm{Y}^{W}}$ is compact and closed, hence, $T_{ \bm{Y}^{W}}$ is finite.
		
	\end{proof}

	\section{Matrix variate symmetric Laplace distribution}
	In this section, we explore the maximum likelihood estimation problem of the matrix variate symmetric Laplace distribution and relate the maximization of likelihood to the norm minimization over a group. Similar to the multivariate symmetric Laplace case, here also a connection between the properties of the likelihood function of the matrix variate symmetric Laplace distribution and the stability of data with respect to the group action is established. In this section again we take the field $\mathbb{R}$ (not $\mathbb{C}$), and $GL_{m}$ and $SL_{m}$ are used for $GL_{m}(\mathbb{R})$ and $SL_{m}(\mathbb{R})$, respectively. 
	
	\begin{definition}[Matrix variate symmetric Laplace distribution \cite{YS}]
		\label{def:5.1}
		A random matrix $\bm{X}$ of order $p\times q$ is said to have a matrix variate symmetric Laplace distribution with parameters $\bm{\Sigma}_{1}\in \mathbb{R}^{p\times p}$ and $\bm{\Sigma}_{2}\in \mathbb{R}^{q\times q}$ (positive definite matrices) if $vec(\bm{X})\sim \mathcal{SL}_{pq}(\bm{\Sigma}_{2} \otimes \bm{\Sigma}_{1})$. This distribution is denoted as $\bm{X} \sim \mathcal{MSL}_{p,q}(\bm{\Sigma}_{1},\bm{\Sigma}_{2})$.
	\end{definition} 
	
	The probability density function of a symmetric Laplace distributed random matrix $\bm{X} \in \mathbb{R}^{p\times q}$ with location parameter zero and parameters $\bm{\Sigma}_{1}\in \mathbb{R}^{p\times p}$ and $\bm{\Sigma}_{2}\in \mathbb{R}^{q\times q}$ (positive definite matrix), is
	\begin{multline}
		\label{eq:3}
		f_{\bm{X}}(\bm{x})=\frac{2}{(2\pi)^\frac{pq}{2} \begin{vmatrix} \bm{\Sigma}_{2}\end{vmatrix}^{p/2} \begin{vmatrix} \bm{\Sigma}_{1}\end{vmatrix}^{q/2}} \left(\frac{\operatorname{tr}\left(\bm{\Sigma}_{2}^{-1} \bm{x}^\top \bm{\Sigma}_{1}^{-1} \bm{x}\right)} {2}\right)^{\frac{\nu}{2}} \\  K_{\nu} \left( \sqrt{ 2 \operatorname{tr}\left(\bm{\Sigma}_{2}^{-1} \bm{x}^\top \bm{\Sigma}_{1}^{-1} \bm{x}\right)}\right),
	\end{multline}
	where $\nu =\frac{2-pq}{2}$ and $K_{\nu}$ is the modified Bessel function of the third kind.
	
	Next theorem is a representation of matrix variate symmetric Laplace distribution, which is reproduced from the paper \cite{YS}. 
	\begin{theorem}[Representation \cite{YS}]
		\label{theorem:5.3} If $\bm{Z}\sim \mathcal{MN}_{p,q}(\bm{0}, \bm{\Sigma}_{1}, \bm{\Sigma}_{2})$, $W\sim Exp(1)$ and $\bm{Z}$ and $W$ are independent. Then, the random variable $\bm{X}= \sqrt{W} \bm{Z}$ has a matrix variate symmetric Laplace distribution with probability density function given in \eqref{eq:3}.
	\end{theorem}
	
	\begin{corollary}
		\label{corly:1}
		If $\bm{X} \sim \mathcal{MSL}_{p,q}(\bm{\Sigma}_{1},\bm{\Sigma}_{2})$ and $W \sim Exp(1)$, then
		the joint probability density function of the random matrix $\bm{X}$ and random variable $W$ is
		\begin{equation*}
			f_{\bm{X}, W}(\bm{x}, w)=\frac{\exp(-w)} {(2\pi)^\frac{pq}{2}(w)^\frac{pq}{2} \begin{vmatrix} \bm{\Sigma}_{2}\end{vmatrix}^{p/2} \begin{vmatrix} \bm{\Sigma}_{1}\end{vmatrix}^{q/2}} \exp\left(-\frac{1}{2w} \operatorname{tr}\left( \bm{\Sigma}_{2}^{-1} \bm{x}^\top \bm{\Sigma}_{1}^{-1} \bm{x} \right) \right).
		\end{equation*}
	\end{corollary}
	
	\begin{proof}
		From the above \hyperref[theorem:5.3]{Theorem \ref{theorem:5.3}}, the representation of a matrix variate symmetric Laplace distributed random variable $\bm{X} = \sqrt{W} \bm{Z}$, then 
		\[vec(\bm{X})=\sqrt{W} vec(\bm{Z})\sim \mathcal{SL}_{pq}(\bm{\Sigma}_{2} \otimes \bm{\Sigma}_{1}).\] 
		
		From \hyperref[thm:5.1]{Theorem \ref{thm:5.1}}, the joint probability density function of $vec(\bm{X})$ and $W$ is 
		
		\begin{equation*}
			f_{vec(\bm{X}),W}(vec(\bm{x}),w)= \frac{\exp\left(-w-\frac{\left(vec(\bm{x})\right)^\top  \left(\bm{\Sigma}_{2} \otimes \bm{\Sigma}_{1}\right)^{-1} vec(\bm{x})}{2w}\right) }{(2\pi)^{\frac{pq}{2}} \begin{vmatrix} \bm{\Sigma}_{2} \otimes \bm{\Sigma}_{1} \end{vmatrix}^{\frac{1}{2}}  w^\frac{pq}{2}}.
		\end{equation*}
		
		Using properties of Kronecker product, trace and determinants (see, \cite{GN}, \cite{RC}),
		\[\begin{vmatrix} \bm{\Sigma}_{2} \otimes \bm{\Sigma}_{1} \end{vmatrix} = \begin{vmatrix} \bm{\Sigma}_{2}\end{vmatrix}^{p} \begin{vmatrix} \bm{\Sigma}_{1}\end{vmatrix}^{q},\]
		\begin{equation*}
			\left(vec(\bm{x})\right)^\top \left(\bm{\Sigma}_{2} \otimes \bm{\Sigma}_{1} \right)^{-1} vec(\bm{x})	= \operatorname{tr}\left(\bm{\Sigma}_{2}^{-1} \bm{x}^\top \bm{\Sigma}_{1}^{-1} \bm{x}\right).
		\end{equation*}
		
		The joint probability density function of $\bm{X}$ and $W$ is
		\[f_{\bm{X}, W}(\bm{x}, w)=\frac{\exp(-w)} {(2\pi)^\frac{pq}{2}(w)^\frac{pq}{2} \begin{vmatrix} \bm{\Sigma}_{2}\end{vmatrix}^{p/2} \begin{vmatrix} \bm{\Sigma}_{1}\end{vmatrix}^{q/2}} \exp\left(-\frac{1}{2w} \operatorname{tr}\left( \bm{\Sigma}_{2}^{-1} \bm{x}^\top \bm{\Sigma}_{1}^{-1} \bm{x} \right) \right). \]
	\end{proof}

	\subsection{Matrix variate symmetric Laplace distribution as group model}
	Consider the multivariate symmetric Laplace distribution of dimension $pq$. The matrix Laplace model is a submodel with $\bm{\Sigma}=\bm{\Sigma}_{2} \otimes \bm{\Sigma}_{1}$. Setting $\bm{\Psi}_{1}= {\bm{\Sigma}_{1}}^{-1}$ and $\bm{\Psi}_{2}= {\bm{\Sigma}_{2}}^{-1}$, then the concentration matrix is \[\bm{\Psi}=\bm{\Sigma}^{-1}=\left(\bm{\Sigma}_{2} \otimes \bm{\Sigma}_{1}\right)^{-1}=\left(\bm{\Sigma}_{2} ^{-1}\otimes \bm{\Sigma}_{1} ^{-1}\right)=\bm{\Psi}_{2}\otimes \bm{\Psi}_{1}.\]
	The matrix Laplace group model is constructed using the representation
	\begin{align*}
		\rho : GL_{p} \times GL_{q} & \mapsto GL_{pq},\\
		(A_{1},A_{2} )  &\longmapsto A_{2}\otimes A_{1},
	\end{align*}
	and the matrix variate symmetric Laplace distribution arises as the Laplace group model of $G:= \rho \left(GL_{p} \times GL_{q} \right) \subseteq GL_{pq}$. Thus, the concentration matrices in the Laplace group model are of the form
	\[( A_{2}\otimes A_{1})^\top  \left(A_{2}\otimes A_{1}\right) =\left(A_{2}^\top A_{2}\right) \otimes \left(A_{1}^\top A_{1}\right).\]
	
	\begin{definition}[Matrix Laplace group model]
		The matrix Laplace group model defined by $G\subseteq GL_{pq}$ consists of matrix variate symmetric Laplace distributions with $\bm{\Psi}=\bm{\Psi}_{2}\otimes \bm{\Psi}_{1}\in PD_{pq}(\mathbb{R})$ lying in the set 
		\[\mathcal{L}_{G}=\{( A_{2}\otimes A_{1})^\top  \left(A_{2}\otimes A_{1}\right) \mid (A_{1},A_{2})\in GL_{p} \times GL_{q} \}= \{A^\top A \mid A\in G\}.\]
	\end{definition}
	
	\subsection{Maximum likelihood estimation}
	Let $\underline{\bm{X}}=\left( \bm{X}_1, \bm{X}_2,\ldots , \bm{X}_N \right)$ be random sample from a matrix variate symmetric Laplace distribution $\mathcal{MSL}_{p,q}(\bm{\Sigma}_{1}, \bm{\Sigma}_{2})$.  The likelihood function is 
	\begin{equation*}
		L\left( \bm{\Sigma}_{1}, \bm{\Sigma}_{2} \mid \bm{X}_{1},\cdots,\bm{X}_{N} \right)=f\left(\underline{\bm{X}}; \bm{\Sigma}_{1}, \bm{\Sigma}_{2} \right) = \prod_{i=1}^{N} f(\bm{X}_{i}), 
	\end{equation*}
	where $f(\bm{X}_{i})$ is as in \eqref{eq:3}. Then, the log-likelihood function (up to an additive constant) is 
	\begin{multline*}
		\ell(\bm{\Sigma}_{1}, \bm{\Sigma}_{2}) = -\frac{qN}{2} \log {\begin{vmatrix} \bm{\Sigma}_{1}\end{vmatrix}}-\frac{pN}{2} \log {\begin{vmatrix} \bm{\Sigma}_{2}\end{vmatrix}} +\frac{\nu}{2} \sum_{i=1}^{N} \log \left(\operatorname{tr}\left(\bm{\Sigma}_{2}^{-1}{\bm{X}_{i}}^\top \bm{\Sigma}_{1}^{-1} \bm{X}_{i}\right) \right) \\
		+ \sum_{i=1}^{N}\log K_{\nu} \left( \sqrt{2 \operatorname{tr} \left(\bm{\Sigma}_{2}^{-1} {\bm{X}_{i}}^\top \bm{\Sigma}_{1}^{-1} \bm{X}_{i}\right)}  \right).
	\end{multline*}
	Parameters $\bm{\Sigma}_{1}$ and $\bm{\Sigma}_{2}$ in the argument of $K_{\nu}$, the modified Bessel function of the third kind, makes maximizing the log-likelihood function difficult, as the score equations do not have explicit solutions. So, to compute the MLEs of the parameters $\bm{\Sigma}_{1}$ and $\bm{\Sigma}_{2}$, we are using \hyperref[theorem:5.3]{Theorem \ref{theorem:5.3}}, which is a representation of matrix variate symmetric Laplace random variable $\bm{X}$. We are using the joint density function of $\bm{X}$ and $W$ with the representation $\bm{X}=\sqrt{W}\bm{Z}$ of matrix variate symmetric Laplace random variable $\bm{X}$, where random matrix $\bm{Z}\sim \mathcal{MN}_{p,q}(\bm{0}, \bm{\Sigma}_{1},\bm{\Sigma}_{2})$ and random variable $W\sim Exp(1)$ are independent. 
	
	\begin{lemma}
		\label{lemma:6.3}
		Let $\left(\underline{\bm{X}},\underline{W}\right)=\left(\bm{X}_{1},\bm{X}_{2},\cdots,\bm{X}_{N},W_{1},W_{2},\cdots,W_{N}\right)$ be complete random sample from the joint distribution of $\bm{X}$ and $W$, where $W\sim Exp(1)$ and $\bm{X} \sim \mathcal{MSL}_{p,q}(\bm{\Sigma}_{1},\bm{\Sigma}_{2})$. If $\left(\hat{\bm{\Sigma}_{1}},\hat{\bm{\Sigma}_{2}} \right)$ maximizes the complete data likelihood function, then $\left(\hat{\bm{\Sigma}_{1}},\hat{\bm{\Sigma}_{2}} \right)$ also maximizes the likelihood function of $p\times q$-dimensional symmetric Laplace distribution.
	\end{lemma}
	\begin{proof}
		Let $\left(\underline{\bm{X}},\underline{W}\right)$ be random sample from the joint distribution of $\bm{X}$ and $W$, where $\bm{X} \sim \mathcal{MSL}_{p,q}(\bm{\Sigma}_{1},\bm{\Sigma}_{2})$ and $W\sim Exp(1)$, and the parameter space of $ \mathcal{MSL}_{p,q}(\bm{\Sigma}_{1},\bm{\Sigma}_{2})$ is $\mathcal{P}=PD_{p}(\mathbb{R})\times PD_{q}(\mathbb{R})$. The complete data likelihood function for $\bm{\Sigma}$ is $L_{c}\left( \bm{\Sigma}_{1},\bm{\Sigma}_{2} \mid \underline{\bm{X}},\underline{W} \right) =  f_{c}\left( (\underline{\bm{X}},\underline{W});\bm{\Sigma}_{1},\bm{\Sigma}_{2} \right)$

		The likelihood function of $\bm{\Sigma}_{1},\bm{\Sigma}_{2}$ given the data $\underline{\bm{X}}$, is
		\begin{equation*}
			L\left( \bm{\Sigma}_{1},\bm{\Sigma}_{2}\mid \underline{\bm{X}} \right)=f(\underline{\bm{X}}; \bm{\Sigma}_{1},\bm{\Sigma}_{2}) =\int_{\mathbb{R}^{N}}L_{c}\left( \bm{\Sigma}_{1},\bm{\Sigma}_{2} \mid \underline{\bm{X}},\underline{W} \right) \, d\underline{W}
		\end{equation*}
		
		Now, suppose
		\begin{equation*}
			\arg \max_{\left(\bm{\Sigma}_{1} ,\bm{\Sigma}_{2}\right) \in \mathcal{P}} L_{c}\left( \bm{\Sigma}_{1},\bm{\Sigma}_{2} \mid \underline{\bm{X}},\underline{W} \right)  =\left(\hat{\bm{\Sigma}_{1}},\hat{\bm{\Sigma}_{2}} \right),
		\end{equation*}
		that is, 
		\begin{equation*}
			L_{c}\left( \bm{\Sigma}_{1},\bm{\Sigma}_{2} \mid \underline{\bm{X}},\underline{W} \right)   \le L_{c} \left( \hat{\bm{\Sigma}_{1}},\hat{\bm{\Sigma}_{2}} \mid  (\underline{\bm{X}},\underline{W})\right), \hspace{.9mm} \forall \left(\bm{\Sigma}_{1},\bm{\Sigma}_{2}\right).
		\end{equation*}
		Thus,
		\begin{align*}
			L( \bm{\Sigma}_{1},\bm{\Sigma}_{2} \mid \underline{\bm{X}} )&= \int_{\mathbb{R}^{N}} L_{c}\left( \bm{\Sigma}_{1},\bm{\Sigma}_{2} \mid( \underline{\bm{X}},\underline{W})\right) \, d\underline{W} \\
			& \le \int_{\mathbb{R}^{N}} L_{c}\left( \hat{\bm{\Sigma}_{1}},\hat{\bm{\Sigma}_{2}} \mid (\underline{\bm{X}},\underline{W}) \right) \, d\underline{W} \hspace{.9mm} \forall \left(\bm{\Sigma}_{1},\bm{\Sigma}_{2} \right)\\
			&=L(\hat{\bm{\Sigma}_{1}},\hat{\bm{\Sigma}_{2}} \mid \underline{\bm{X}} ),  \hspace{.9mm} \forall \left(\bm{\Sigma}_{1},\bm{\Sigma}_{2} \right).
		\end{align*} 
		This implies 
		\[\arg \max_{\left(\bm{\Sigma}_{1},\bm{\Sigma}_{2} \right)\in \mathcal{P}} L(\bm{\Sigma}_{1},\bm{\Sigma}_{2} \mid \underline{\bm{X}} ) =\left(\hat{\bm{\Sigma}_{1}},\hat{\bm{\Sigma}_{2}} \right).\]
	\end{proof}
	
	From the above \hyperref[lemma:6.3]{Lemma \ref{lemma:6.3}}, we can use the complete data likelihood function to find the MLE of $\bm{\Sigma}_{1}$ and $\bm{\Sigma}_{2}$. Using the joint probability density function of $\bm{X}$ and $W$ given in \hyperref[corly:1]{Corollary \ref{corly:1}}, the complete data log-likelihood function (up to an additive constant) is
	\begin{multline*}
		\ell_{c}(\bm{\Sigma}_{1}, \bm{\Sigma}_{2})	=-\frac{qN}{2} \log \begin{vmatrix} \bm{\Sigma}_{1}\end{vmatrix}-\frac{pN}{2} \log \begin{vmatrix} \bm{\Sigma}_{2} \end{vmatrix}-\frac{1}{2}\sum_{i=1}^{N}\frac{1}{W_{i}} \operatorname{tr}\left( \bm{\Sigma}_{2}^{-1}{\bm{X}_{i}}^\top \bm{\Sigma}_{1}^{-1} \bm{X}_{i} \right)\\
		-\left(\sum_{i=1}^{N} \left( W_{i} + \frac{pq}{2} \log(W_{i})\right) \right).
	\end{multline*}
	Since the last term does not contain any unknown parameter, it can be ignored for maximization of $\ell_{c}(\bm{\Sigma}_{1}, \bm{\Sigma}_{2})$ with respect to $\bm{\Sigma}_{1}$ and $\bm{\Sigma}_{2}$.
	Therefore, the MLEs of the parameters $\bm{\Sigma}_{1}$ and $\bm{\Sigma_{2}}$ are that values which maximize the complete data log-likelihood function (upto an additive constant)
	\begin{equation*}
		\ell_{c}(\bm{\Sigma}_{1}, \bm{\Sigma}_{2})= -\frac{qN}{2} \log \begin{vmatrix} \bm{\Sigma}_{1}\end{vmatrix}-\frac{pN}{2} \log \begin{vmatrix} \bm{\Sigma}_{2}\end{vmatrix}-\frac{1}{2}\sum_{i=1}^{N}\frac{1}{W_{i}} \operatorname{tr}\left( \bm{\Sigma}_{2}^{-1}{\bm{X}_{i}}^\top \bm{\Sigma}_{1}^{-1} \bm{X}_{i} \right).
	\end{equation*}
	
	Consider $\bm{\Psi}_{1}= {\bm{\Sigma}_{1}}^{-1}$ and $\bm{\Psi}_{2}= {\bm{\Sigma}_{2}}^{-1}$, the log-likelihood function for the matrix Laplace model can be written as 
	\begin{align}
		\label{eq:14}
		\ell_{c}(\bm{\Psi}_{1}, \bm{\Psi}_{2})&= qN\log \begin{vmatrix} \bm{\Psi}_{1}\end{vmatrix}+pN \log \begin{vmatrix} \bm{\Psi}_{2}\end{vmatrix}-\sum_{i=1}^{N}\frac{1}{W_{i}} \operatorname{tr}\left( \bm{\Psi}_{2}{\bm{X}_{i}}^\top \bm{\Psi}_{1} \bm{X}_{i} \right)\\
		&= N\log \left( \begin{vmatrix}\bm{\Psi}_{2}\otimes \bm{\Psi}_{1} \end{vmatrix}\right) -\sum_{i=1}^{N}\frac{1}{W_{i}} \operatorname{tr}\left( \bm{\Psi}_{2}{\bm{X}_{i}}^\top \bm{\Psi}_{1} \bm{X}_{i} \right).
	\end{align}
	
	An MLE is a concentration matrix $\hat{\bm{\Psi}}_{2}\otimes \hat{\bm{\Psi}}_{1} \in PD_{pq}(\mathbb{R})$ that maximizes the complete data log-likelihood function \eqref{eq:14}.
	
	\subsection{Equivalence of optimization problem}
	In this subsection, the maximization of the complete data log-likelihood function is compared to the minimization of the norm through the group actions, and the correspondence between stability of data with respect to the group action and the properties of the likelihood function is proved.
	
	For a random sample $\left(\bm{X}_{1},\bm{X}_{2},\cdots,\bm{X}_{N},W_{1},W_{2},\cdots,W_{N}\right)$ from the joint distribution of $\bm{X}$ and $W$, let 
	\begin{equation}
		\label{eq:7.4}
		\bm{X}^{W}=\left(\frac{1}{\sqrt{W_{1}}}\bm{X}_{1},\frac{1}{\sqrt{W_{2}}} \bm{X}_{2},\cdots, \frac{1}{\sqrt{W_{N}}}\bm{X}_{N}\right)\in (\mathbb{R}^{p\times q})^{N}.
	\end{equation}
	Then, the left-right action of the group $G \subseteq GL_{p} \times GL_{q}$ on $(\mathbb{R}^{p\times q})^N$ is defined
	$\forall A\in G$ and $ \bm{X}^{W}\in (\mathbb{R}^{p\times q})^{N}$ as:  \[A.\bm{X}^{W}=\left(\frac{1}{\sqrt{W_{1}}}A_{1}\bm{X}_{1}A_{2}^\top,\frac{1}{\sqrt{W_{2}}}A_{1}\bm{X}_{2}A_{2}^\top,\cdots,\frac{1}{\sqrt{W_{N}}}A_{1}\bm{X}_{N}A_{2}^\top \right).\]
	
	\begin{remark}
		If $\left(\underline{\bm{X}},\underline{W}\right)=\left(\bm{X}_{1},\bm{X}_{2},\cdots,\bm{X}_{N},W_{1},W_{2},\cdots,W_{N}\right)$ is the random sample from the joint distribution of $\bm{X}$ and $W$, then \[ \bm{X}^{W}=\left(\frac{1}{\sqrt{W_{1}}} \bm{X}_{1},\frac{1}{\sqrt{W_{2}}}  \bm{X}_{2},\cdots, \frac{1}{\sqrt{W_{N}}} \bm{X}_{N}\right)\] is a statistic (data), and the MLE of the parameter given $\bm{X}^{W}$ is equivalent to stating that the MLE of the parameter given $\left(\underline{\bm{X}},\underline{W}\right)$. Here, we talk about the stability of the data $\bm{X}^{W}$ rather than the random sample $\left(\underline{\bm{X}},\underline{W}\right)$, with respect to the group action.
	\end{remark}
	
	The norm can be written as 
	\begin{equation*}
		\|A\cdot \bm{X}^{W}\|^2 =\sum_{i=1}^{N} \left(\frac{1}{W_{i}} \| A_{1}\bm{X}_{i}A_{2}^\top\|_F^2\right)
		=\sum_{i=1}^{N} \frac{1}{W_{i}}	\operatorname{tr} \left( \left(A_{1}\bm{X}_{i}A_{2}^\top\right)^\top\left(A_{1}\bm{X}_{i}A_{2}^\top\right) \right).
	\end{equation*}
	
	\begin{equation}
		\label{eq:16}
		\|A\cdot \bm{X}^{W}\|^2 = \sum_{i=1}^{N} \frac{1}{W_{i}}	\operatorname{tr} \left(A_{2}^\top A_{2}\bm{X}_i^\top A_{1}^\top A_{1} \bm{X}_{i} \right).
	\end{equation}
	
	To relate the norm minimization to maximum likelihood estimation, consider the complete data log-likelihood function \eqref{eq:14}
	\begin{equation*}
		\ell_{c}(\bm{\Psi}_{1}, \bm{\Psi}_{2})= N\log \left( \begin{vmatrix}\bm{\Psi}_{2}\otimes \bm{\Psi}_{1} \end{vmatrix}\right) -\sum_{i=1}^{N}\frac{1}{W_{i}} \operatorname{tr}\left( \bm{\Psi}_{2}{\bm{X}_{i}}^\top \bm{\Psi}_{1} \bm{X}_{i} \right).
	\end{equation*}
	
	Since $\bm{\Psi}_{2}\otimes \bm{\Psi}_{1}\in \mathcal{M}_{G}$, then $\bm{\Psi}_{2}\otimes \bm{\Psi}_{1}$ can be written as $A^\top A$, where $A\in G \subseteq GL_{pq}$, then 
	
	\begin{equation}
		\label{eq:17}
		\ell_{c}(A^\top A)=N \log \begin{vmatrix}A^\top A \end{vmatrix} -\sum_{i=1}^{N} \frac{1}{W_{i}}	\operatorname{tr} \left( A_{2}^\top A_{2}\bm{X}_i^\top A_{1}^\top A_{1} \bm{X}_{i} \right).
	\end{equation}
	
	Comparing \eqref{eq:16} and \eqref{eq:17}, maximization the log-likelihood
	over $\mathcal{M}_{G}$ is equivalent to minimizing
	\[ -\ell_{c} (A^\top A)=- N \log\begin{vmatrix}A^\top A \end{vmatrix}+\|A\cdot \bm{X}^{W}\|^2 \]
	over $A\in G$.

	The subgroup $G\subseteq GL_{pq}$ is Zariski closed, self-adjoint group, and closed under non-zero scalar multiples. Therefore our results from the previous section is apply to the action of $G_{SL}^+$. So, it is convenient to directly work with the left-right action of $SL_{p}\times SL_{q}$.

	\begin{theorem}
		\label{thm:11}
		Let $\bm{X}^{W} \in (\mathbb{R}^{p\times q})^N$ be as defined in \eqref{eq:7.4}. If the group $G\subseteq GL_{pq}$ is a Zariski closed self adjoint, which is closed under non-zero scalar multiplication, then
		\begin{equation}
			\sup_{A\in G} \left(\ell_{c}(A^\top A) \right) \equiv
			\inf _{\alpha \in {\mathbb{R}}_{>0}} \left( \alpha \left( 	\inf_{B\in SL_{p}\times SL_{q}} \| B\cdot \bm{X}^{W}\|^2  \right)  -pqN\log \alpha \right).
		\end{equation}
		The MLEs, if they exists are the matrices $\alpha B_{2}^\top B_{2}\otimes B_{1}^\top B_{1}$, where $B=(B_{1},B_{2})$ minimizes $\|B\cdot \bm{X}^{W}\|$ under the action of $SL_{p}\times SL_{q}$ on $(\mathbb{R}^{p\times q})^N$ and $\alpha \in {\mathbb{R}}_{>0}$ is the unique value minimizing the outer infimum. 
		
		If $\alpha B_{2}^\top B_{2}\otimes B_{1}^\top B_{1}$, where $B=(B_{1},B_{2})$ is an MLE given $\left(\underline{\bm{X}},\underline{W}\right)$, then all MLEs given $\left(\underline{\bm{X}},\underline{W}\right)$ are of the form $\alpha S_{2}^\top B_{2}^\top B_{2} S_{2} \otimes S_{1}^\top  B_{1}^\top B_{1} S_{1}$, where $(S_{1},S_{2})$ is in the  $SL_{p}\times SL_{q}$ stabilizer of $\bm{X}^{W}$. 
		
		The stability under the left-right action of $SL_{p}\times SL_{q}$ on $(\mathbb{R}^{p \times q})^ N$ is related to maximum likelihood estimation of the matrix variate symmetric Laplace distribution over $\mathcal{M}_{G}$ as follows:
		\renewcommand{\labelenumi}{(\alph{enumi})}
		\begin{enumerate}
			\item $\bm{X}^{W}$ is unstable $\iff$ $\ell_{c}(\bm{\Psi}_{1}, \bm{\Psi}_{2})$ is not bounded from above,
			\item $\bm{X}^{W}$ is semistable $\iff$ $\ell_{c}(\bm{\Psi}_{1},\bm{\Psi}_{2})$ is bounded from above,
			\item $\bm{X}^{W}$ is polystable $\iff$ MLE exists,
			\item  $\bm{X}^{W}$ is stable $\implies$ unique MLE exists.
		\end{enumerate}
	\end{theorem}

	\begin{proof}
		The subgroup $H= \rho \left( SL_{p}\times SL_{q} \right) \subseteq G$ is Zariski closed, self-adjoint. Thus, the Theorems \ref{prop:2} and \ref{prop:5}, and the \hyperref[thm:6 ]{Theorem \ref{thm:6 }} apply to $H$ as well. Furthermore, the kernel of the map $\rho $ with restricted to $SL_{p} \times SL_{q}$, is finite. Hence, for the stability notions $SL_{p} \times SL_{q}$ is equivalent to the image map $H$, so $SL_{p} \times SL_{q}$  can be considered instead of its image $H$.
	\end{proof}
	
	The converse part of \hyperref[thm:11]{Theorem \ref{thm:11}} (d) also holds for the complex matrix Laplace group model, as for the complex Laplace group model given in \hyperref[thm:10]{Theorem \ref{thm:10}}. 
	
	The next example exhibits that the converse part of \hyperref[thm:11]{Theorem \ref{thm:11}} (d) is not true for the real Laplace group model, that is, a unique MLE exists given $\bm{X}^{W}$, but the stabilizer of $\bm{X}^{W}$ is an infinite set.
	
	\begin{example}
		\label{ex:2}
		Let $p=q=N=2$ and the random sample \[\left( \underline{\bm{X}},\underline{W} \right)=\left( \bm{X}_{1}, \bm{X}_{2}, W_{1},W_{2}\right)=\left( \begin{bmatrix}
			1 &0\\0&1
		\end{bmatrix}, \begin{bmatrix}
			1 &-1\\1&1
		\end{bmatrix}, 1, 2 \right).\] 
	Then \[\bm{X}^{W}=\left( \frac{1}{\sqrt{W_{1}}}\bm{X}_{1},  \frac{1}{\sqrt{W_{2}}}\bm{X}_{2} \right)\in (\mathbb{R}^{2\times2})^{2} ,\]
		with
		\[\frac{1}{\sqrt{W_{1}}} \bm{X}_{1}=\begin{bmatrix}
			1 &0\\0&1
		\end{bmatrix}, \frac{1}{\sqrt{W_{2}}} \bm{X}_{2}=\frac{1}{\sqrt{2}} \begin{bmatrix}
			1 &-1\\1&1
		\end{bmatrix}.\] 
		
		We will show that the MLE given $\bm{X}^{W}$ is unique, but the stabilizer of $\bm{X}^{W}$ is an infinite set. First we will show that $\bm{X}^{W}$ is polystable under the left-right action of $SL_{2}\times SL_{2}$. Any matrix in $SL_{2}$ has Frobenius norm at least $\sqrt{2}$, and $\|\frac{1}{\sqrt{W_{1}}}\bm{X}_{1} \|^{2}=2$ and $\|\frac{1}{\sqrt{W_{2}}}\bm{X}_{2} \|^{2}=2$, therefore $\bm{X}^{W}=\left( \frac{1}{\sqrt{W_{1}}}\bm{X}_{1},  \frac{1}{\sqrt{W_{2}}}\bm{X}_{2} \right)$ has minimal norm in its orbit. Then, by (a) and (d) of the Kempf-Ness \hyperref[thm:1]{Theorem \ref{thm:1}}, $\bm{X}^{W}$ is polystable. The stabilizer of $\bm{X}^{W}$ is $T_{\bm{X}^{W}}$, which is the set \[\biggl\{ (S_{1},S_{2})\in SL_{2}\times SL_{2} \bigg| \, \frac{1}{\sqrt{W_{1}}}S_{1}\bm{X}_{1}S_{2}^\top =\frac{1}{\sqrt{W_{1}}} \bm{X}_{1}, \frac{1}{ \sqrt{W_{2}}}S_{1}\bm{X}_{2}S_{2}^\top =\frac{1}{\sqrt{W_{2}}}\bm{X}_{2}\biggr\}.\] 
		
		Now,  $\frac{1}{\sqrt{W_{1}}}S_{1}\bm{X}_{1}S_{2}^\top =\frac{1}{\sqrt{W_{1}}} \bm{X}_{1}\implies S_{1}S_{2}^\top =\bm{I} \implies S_{2}^\top =S_{1}^{-1}$, and $\frac{1}{ \sqrt{W_{2}}}S_{1}\bm{X}_{2}S_{2}^\top =\frac{1}{\sqrt{W_{2}}}\bm{X}_{2} \implies S_{1}\bm{X}_{2}S_{1}^{-1} =\bm{X}_{2}\implies S_{1}\bm{X}_{2}= \bm{X}_{2} S_{1}$, after solving this, we get
		\[S_{1}=\begin{bmatrix}
			a&b\\-b&a
		\end{bmatrix}, a^{2}+b^{2}=1,\]
		that is, $S_{1}$ is an orthogonal matrix, and $S_{2}=(S_{1}^{-1} )^\top=S_{1}$. Hence, the stabilizer   $T_{\bm{X}^{W}}= \{ (S_{1},S_{1}) \mid S_{1}\in SO_{2}\}$, where $SO_{2}$ is the set of orthogonal matrices in $SL_{2}$.
		
		Since $\bm{X}^{W}$ has minimal norm in its orbit, then by \hyperref[thm:11]{Theorem \ref{thm:11}}, $\alpha \bm{I}\otimes \bm{I}$ is an MLE given $\bm{X}^{W}$. All the MLEs given $\bm{X}^{W}$ are $\alpha S_{2}^\top \bm{I} S_{2} \otimes S_{1}^\top \bm{I} S_{1}$, where $(S_{1},S_{2})\in T_{\bm{X}^{W}}$. Therefore, the MLE $\alpha \bm{I} \otimes \bm{I}$ given $\bm{X}^{W}$ is unique. But the stabilizer $T_{\bm{X}^{W}}$ is an infinite set, that is, $\bm{X}^{W}$ is not stable. 
		
		For complex matrix Laplace group model, $T_{\bm{X}^{W}}= \{ (S_{1},S_{1})\mid S_{1}\in SO_{2}(\mathbb{C})\}$, and from \hyperref[prop:9]{Theorem \ref{prop:9}}, all the MLEs given $\bm{X}^{W}$ are $\alpha S_{2}^{*} S_{2} \otimes S_{1}^{*}S_{1}$, where $(S_{1},S_{2}) \in T_{\bm{X}^{W}}$. Hence, infinitely many MLEs exists.
	\end{example}
	
	In the following example, we will demonstrate all the stability conditions of data with respect to the left-right action of $SL_{2}\times SL_{2}$ on $(\mathbb{R}^{2\times 2})^{N}$ that may arise.
	
	\begin{example}
		Let $p=q=2$ and we will show the stability of data under the left-right action of $SL_{2}\times SL_{2}$ on $(\mathbb{R}^{2\times 2})^{N}$. Consider the random sample
		\[ \bm{X}_{1}=\begin{bmatrix}
			0 &0\\2&0
		\end{bmatrix}, \bm{X}_{2}= \begin{bmatrix}
			1 &0\\0&1
		\end{bmatrix},\bm{X}_{3}= \begin{bmatrix}
			1 &-1\\1&1
		\end{bmatrix}, \bm{X}_{4}=\begin{bmatrix}
			\sqrt{3}&0\\0&-\sqrt{3}
		\end{bmatrix},\]
		and $W_{1}=4,W_{2}=1,W_{3}=2,W_{4}=3$. Then,
		\[\frac{1}{\sqrt{W_{1}}} \bm{X}_{1}=\begin{bmatrix}
			0 &0\\1&0
		\end{bmatrix}, \frac{1}{\sqrt{W_{2}}} \bm{X}_{2}= \begin{bmatrix}
			1 &0\\0&1
		\end{bmatrix},\]
		\[\frac{1}{\sqrt{W_{3}}} \bm{X}_{3}=\frac{1}{\sqrt{2}} \begin{bmatrix}
			1 &-1\\1&1
		\end{bmatrix}, \frac{1}{\sqrt{W_{4}}} \bm{X}_{4}= \begin{bmatrix}
			1 &0\\0&-1
		\end{bmatrix}.\]
		
		\renewcommand{\labelenumi}{(\alph{enumi})}
		\begin{enumerate}
			\item Consider $\bm{X}^{W}=\frac{1}{\sqrt{W_{1}}}\bm{X}_{1}$, and 
			\[A_{1}=\begin{bmatrix}
				t^{-1} &0\\0&t
			\end{bmatrix} \in SL_{2}, \bm{I}\in SL_{2}.\]
			Then, $(A_{1},\bm{I}) \cdot \bm{X}^{W}= A_{1} \frac{1}{\sqrt{W_{1}}}\bm{X}_{1} \bm{I}^\top = \begin{bmatrix}
				0 &0\\t&0
			\end{bmatrix},$ then, $\|(A_{1},\bm{I}) \cdot \bm{X}^{W}\|=0$, as $t$ tends to $0$, that is, $0$ is in the orbit closure of $\bm{X}^{W}$. Hence, $\bm{X}^{W}=\frac{1}{\sqrt{W_{1}}}\bm{X}_{1}$ is unstable.
			
			\item Consider $\bm{X}^{W}=\left( \frac{1}{\sqrt{W_{1}}}\bm{X}_{1}, \frac{1}{\sqrt{W_{2}}}\bm{X}_{2} \right)$, and the orbit of $\bm{X}^{W}$ is contained in the set $\{ (S, B)\mid S\ne \bm{0}, B\in SL_{2}\}$, for any $(S,B)$ is in the orbit of $\bm{X}^{W}$, where $S\ne \bm{0}$ and $B\in SL_{2}$
			\[ \|(S,B)\|^{2}=\|S\|^{2}+\|B\|^{2}\ge \|S\|^{2}+2>2,\]
			that is, $\bm{X}^{W}$ is semistable. Let 
			\[A_{1}=\begin{bmatrix}
				t^{-1} &0\\0&t
			\end{bmatrix} , 
			A_{2}=\begin{bmatrix}
				t &0\\0&t^{-1}
			\end{bmatrix}, \] then $ (A_{1},A_{2})\cdot \bm{X}^{W}=	\left( A_{1} \frac{1}{\sqrt{W_{1}}} \bm{X}_{1}A_{2}^\top, A_{1} \frac{1}{\sqrt{W_{2}}} \bm{X}_{2}A_{2}^\top \right)=(\bm{0},\bm{I}),$ as $t$ tends to $0$, and the norm $\| (\bm{0},\bm{I})\|^{2}=2$, which is minimal norm but does not attained by any element in the orbit of $\bm{X}^{W}$. Hence, $\bm{X}^{W}=\left( \frac{1}{\sqrt{W_{1}}}\bm{X}_{1}, \frac{1}{\sqrt{W_{2}}}\bm{X}_{2} \right)$ is not polystable.
			
			\item The data $\bm{X}^{W}=\frac{1}{\sqrt{W_{2}}}\bm{X}_{2}=\bm{I}$ is of minimal norm, so by the Kempf-Ness \hyperref[thm:1]{Theorem \ref{thm:1}} (a) and (d), $\bm{X}^{W}$ is polystable. Thus by \hyperref[thm:11]{Theorem \ref{thm:11}}, $\alpha \bm{I} \otimes \bm{I}$ is an MLE given $\bm{X}^{W}$. The stabilizer of $\bm{X}^{W}$ is
			\begin{align*}
				T_{\bm{X}^{W}}&=\{(A_{1},A_{2})\in SL_{2} \times SL_{2}\mid (A_{1}, A_{2})\cdot \bm{X}^{W}=\bm{X}^{W}\}\\
				&=\{ (A_{1},A_{2})\in SL_{2} \times SL_{2}\mid A_{1} \bm{I} A_{2}^\top =\bm{I}\}\\
				&=\{(A_{1},A_{2})\in SL_{2} \times SL_{2}\mid A_{2}=A_{1}^{-\top}\}\\
				&=\{(A,A^{-\top})\mid A\in SL_{2}\}.
			\end{align*} 
			So, the stabilizer of $\bm{X}^{W}$ is an infinite set, that is, $\bm{X}^{W}$ is not stable. There are infinitely many MLEs exists, and all the MLEs given $\bm{X}^{W}$ are of the form $\alpha A^\top A \otimes A^{-1} A^{-\top}$ for $A\in SL_{2}$.
			
			\item Consider $\bm{X}^{W}=\left( \frac{1}{\sqrt{W_{2}}}\bm{X}_{2}, \frac{1}{\sqrt{W_{3}}}\bm{X}_{3}, \frac{1}{\sqrt{W_{4}}}\bm{X}_{4} \right)$. The orbit of $\bm{X}^{W}$ is contained in the set $\{(S_{1},S_{2},S_{3}) \mid S_{1},S_{2} \in SL_{2}, \det(S_{3})=-1 \}$. $\bm{X}^{W}$ is of minimal norm in its orbit using the same argument of \hyperref[ex:2]{Example \ref{ex:2}}. Hence, by the Kempf-Ness \hyperref[thm:1]{Theorem \ref{thm:1}} (a) and (d), $\bm{X}^{W}$ is polystable. As discussed in \hyperref[ex:2]{Example \ref{ex:2}}, the stabilizer of $\bm{X}^{W}$ is contained in the set $\{(S_{1},S_{1}) \mid S_{1}\in SO_{2}\}$, where $SO_{2}$ is the set of orthogonal matrices in $SL_{2}$ and  the matrix $S_{1}$ is of the form  
			\[ S_{1}=\begin{bmatrix}
				a&b\\-b&a
			\end{bmatrix}, a^{2}+b^{2}=1.\] Now, $S_{1}\frac{1}{\sqrt{W_{4}}} \bm{X}_{4} S_{1}^\top=\frac{1}{\sqrt{W_{4}}} \bm{X}_{4} \implies S_{1}\frac{1}{\sqrt{W_{4}}} \bm{X}_{4}= \frac{1}{\sqrt{W_{4}}} \bm{X}_{4}S_{1}^{-\top} \implies S_{1}\frac{1}{\sqrt{W_{4}}} \bm{X}_{4}= \frac{1}{\sqrt{W_{4}}} \bm{X}_{4}S_{1}$. After solving this, we get $b=0$ and $a^{2}=1$, that is, $S_{1}=\pm \bm{I}$, and the stabilizer of $\bm{X}^{W}$ is $T_{\bm{X}^{W}}= \{(\bm{I},\bm{I}), (-\bm{I},-\bm{I})\}$, a finite set. Hence, $\bm{X}^{W}$ is stable.
		\end{enumerate}
	\end{example}

	\section{Conclusion}
	In this paper, the boundedness of the likelihood function, as well as the existence and uniqueness of maximum likelihood estimators of the parameters of the multivariate and matrix variate symmetric Laplace distributions, are discussed. It is important to note that these distributions belong to the non-exponential family of distributions. To achieve this, we have utilized the representations of the multivariate and matrix variate symmetric Laplace random variables as given in \hyperref[thm:5.1]{Theorem \ref{thm:5.1}} and \hyperref[theorem:5.3]{Theorem \ref{theorem:5.3}}. We demonstrate that the MLE can be derived from the joint distribution established in  \hyperref[lemma:5.2]{Lemma \ref{lemma:5.2}} and \hyperref[lemma:6.3]{Lemma \ref{lemma:6.3}}. 
	
	The boundedness of the likelihood function and the existence and uniqueness of the MLEs of the parameters in the Laplace group model and the matrix Laplace model have been achieved through group actions, where the Laplace group model is defined in \hyperref[def:7.1]{Definition \ref{def:7.1}} and the matrix Laplace model is a submodel of the Laplace group model. We have shown that the maximizing the complete data log likelihood functions of these models are equivalent to minimizing the norm over the group $G$ through the group action. The study of the group action on a set or vector space focuses on identifying the properties that remain unchanged. The key concept of studying the group actions is the orbit of a point, which consists of all the points obtained by acting the group elements on that point. The stability notions of the data are studied using the minimum norm along its orbits. We have connected the stability of data with respect to the group action to the likelihood properties of the multivariate and matrix variate symmetric Laplace distributions, and hence provided the results for the existence and uniqueness of the MLEs.


\begin{thebibliography}{99}
		
		\bibitem{Aldrich} J. Aldrich, {RA} {F}isher and the making of maximum likelihood 1912-1922. \emph{Statistical Science}, \textbf{12}(3), (1997) 162-176.
		\bibitem{amendola2021invariant}C. Améndola, K. Kohn, P. Reichenbach, and A. Seigal, Invariant theory and scaling algorithms for maximum likelihood estimation, \emph{SIAM Journal on Applied Algebra and Geometry}, \textbf{5}(2), (2021) 304-337.
	    \bibitem{AKRS} C. Améndola, K. Kohn, P. Reichenbach, and A. Seigal, Toric invariant theory for maximum likelihood estimation in log-linear models, \emph{Algebraic Statistics}, \textbf{12}(2), (2021) 187–211.
	    \bibitem{B}F. Bowman, \emph{Introduction to Bessel Functions}, Courier Corporation, (2012).
	    
	    \bibitem{GRL} G Casella and RL Berger, \emph{Statistical Inference}, The Wadsworth Group, Belmont, CA, (2002).
	  
		\bibitem{EKL}T. Eltoft, T. Kim, and T. W. Lee, On the multivariate Laplace distribution, \emph{IEEE Signal Processing Letters}, \textbf{13}(5), (2006) 300-303.
		\bibitem{FM}K. Fragiadakis and S. G. Meintanis, Goodness-of-fit tests for multivariate Laplace distributions, \emph{Mathematical and Computer Modelling}, \textbf{53}(5-6), (2011) 769-779.
		\bibitem{GN}A. K. Gupta and D. K. Nagar,  \emph{Matrix Variate Distributions}, Chapman and Hall/CRC (2018).
		\bibitem{RC}R. A. Horn and C. R. Johnson, \emph{Topics in matrix analysis}, Cambridge University Press, Cambridge, England, (1994).
		\bibitem{KS}T. Kollo and M. S. Srivastava, Estimation and testing of parameters in multivariate Laplace distribution, \emph{Communications in Statistics-Theory and Methods}, \textbf{33}(10), (2005) 2363-2387.
		\bibitem{KKP}S. Kotz, T. Kozubowski, and K. Podgórski, \emph{The Laplace distribution and Generalizations: a Revisit with Applications to Communications, Economics, Engineering, and Finance}, Springer Science and Business Media, (2001).
		\bibitem{KMP}T. Kozubowski, S. Mazur, and K. Podgórski, Matrix variate generalized asymmetric Laplace distributions, \emph{Theory of Probability and Mathematical Statistics} \textbf{109}, (2023) 55-80.
		\bibitem{KP}T. J. Kozubowski and K. Podgórski, Asymmetric Laplace laws and modeling financial data, \emph{Mathematical and Computer Modelling}, \textbf{34}(9-11), (2001) 1003-1021.
		\bibitem{MK}GJ. McLachlan and T. Krishnan, \emph{The EM Algorithm and Extensions}, John Wiley and Sons, (2007).
		\bibitem{myung}I. J. Myung, Tutorial on maximum likelihood estimation, \emph{Journal of mathematical Psychology}, \textbf{47}(1), (2003) 90-100.
		\bibitem{OLBC} F. W. Olver, D. W. Lozier, R. F. Boisvert, and  C. W. Clark, \emph{NIST Handbook of Mathematical Functions}, Cambridge University Press (2010).
		\bibitem{Vi}H. Visk, On the parameter estimation of the asymmetric multivariate Laplace distribution, \emph{Communications in Statistics-Theory and Methods}, \textbf{38}(4), (2009) 461-470.
		\bibitem{W}G. N. Watson, \emph{A Treatise on the Theory of Bessel Functions}, The University Press, New York, (1922).
		\bibitem{YS} P. Yadav and T. Srivastava, Maximum Likelihood Estimation of the Parameters of Matrix Variate Symmetric Laplace Distribution, arXiv preprint arXiv:2502.04118 (2025).
		\bibitem{Y}Y. Yurchenko, Matrix variate and tensor variate Laplace distributions, arXiv preprint arXiv:2104.05669 (2021).

	\end{thebibliography}
\end{document}